\documentclass[a4paper,12pt]{amsart}
\usepackage[utf8]{inputenc}
\usepackage[T1]{fontenc}
\usepackage[UKenglish]{babel}
\usepackage[margin=18mm]{geometry}
\usepackage{color}
\usepackage{bbm}

\usepackage{graphicx}

\usepackage{amsmath,amssymb,amsfonts,amsthm}
\usepackage{mathrsfs,eucal,dsfont}
\usepackage{verbatim,enumitem}

\usepackage{hyperref,url}
\usepackage{amssymb, amsmath, amsthm, amscd}

\usepackage{eucal,mathrsfs,dsfont}
\usepackage{color}

\newcommand{\e}{\text{e}}
\newcommand{\R}{\mathds R}

\newcommand{\Pp}{\mathds P}
\newcommand{\Ee}{\mathds E}

\newcommand{\I}{\mathds 1}

\newcommand{\rd}{{\mathds R^d}}

\def\<{\langle}
\def\>{\rangle}

\newtheorem{theorem}{Theorem}[section]
\newtheorem{lemma}[theorem]{Lemma}

\numberwithin{equation}{section}

\theoremstyle{definition}

\newtheorem{example}[theorem]{Example}
\newtheorem{remark}[theorem]{Remark}

\begin{document}
\allowdisplaybreaks
\title[Strict Kantorovich contractions for Markov chains] {\bfseries
Strict Kantorovich contractions for Markov chains and Euler schemes with general noise}

\author{Lu-Jing Huang\qquad Mateusz B. Majka \qquad  Jian Wang}

\thanks{\emph{L.-J. Huang:}
 School of Mathematics and Statistics, Fujian Normal University, 350007 Fuzhou, P.R. China.
   \texttt{lujingh@yeah.net}}

\thanks{\emph{M.\ B.\ Majka:}
School of Mathematical and Computer Sciences, Heriot-Watt University, Edinburgh, EH14 4AS, UK. \texttt{m.majka@hw.ac.uk}}

\thanks{\emph{J. Wang:}
School of Mathematics and Statistics \& Fujian Key Laboratory of Mathematical
Analysis and Applications (FJKLMAA) \& Center for Applied Mathematics of Fujian Province (FJNU), Fujian Normal University, 350007 Fuzhou, P.R. China.
   \texttt{jianwang@fjnu.edu.cn}}

\date{}
\maketitle

\begin{abstract}
We study contractions of Markov chains on general metric spaces with respect to some carefully designed distance-like functions, which are comparable to the total variation and the standard $L^p$-Wasserstein distances for $p \ge 1$. We present explicit lower bounds of the corresponding contraction rates. By employing the refined basic coupling and the coupling by reflection, the results are applied to Markov chains whose transitions include additive stochastic noises that are not necessarily isotropic. This can be useful in the study of Euler schemes for SDEs driven by L\'evy noises. In particular, motivated by recent works on the use of heavy tailed processes in Markov Chain Monte Carlo, we show that chains driven by the $\alpha$-stable noise can have better contraction rates than corresponding chains driven by the Gaussian noise, due to the heavy tails of the $\alpha$-stable distribution.


\medskip


\end{abstract}
\allowdisplaybreaks

\section{Introduction}\label{section1}
Let $(S,d)$ be a separable metric space. For any probability measures $\mu$ and $\nu$ on $S$, the Kantorovich distance  ($L^1$-Wasserstein distance) with respect to the metric $d$ is defined by
\begin{equation}\label{eq:defWasserstein}
W_d(\mu,\nu)=\inf_{\xi\sim\mu,\eta\sim\nu}\Ee[d(\xi,\eta)],
\end{equation}
where the infimum is taken over all
pairs of random variables $(\xi,\eta)$ defined on a common probability
space $(\Omega, \mathcal F)$ such that $\xi$ (resp.\ $\eta$) is distributed as $\mu$ (resp.\ $\nu$). A Markov chain with the transition kernel $p(x,dz)$ on $(S,d)$ is called strictly Kantorovich contractive with respect to $d$, if there exists a constant $c\in (0,1)$ such that
\begin{equation}\label{kant-cont}
W_d(p(x,\cdot),p(y,\cdot))\le (1-c)d(x,y),\quad x,y\in S.
\end{equation}
Strict Kantorovich contractivity goes back to the famous work \cite{D70} by Dobrushin,
and \eqref{kant-cont} is also known as the ``Dobrushin uniqueness condition''.
It is closely related to the ergodicity of Markov processes, particle systems or other random dynamical systems, see e.g. \cite{Chen04, DGW04} and the references therein. In particular, strict Kantorovich contractivity provides an estimate for the spectral gap, cf. \cite{CW94}.
Ollivier in \cite{O09} introduced the
concept of the
``Ricci curvature'' of Markov chains on metric spaces,
according to which a
 Markov chain has a ``positive Ricci curvature'' if \eqref{kant-cont} holds true. L\'evy-Gromov-like Gaussian concentration theorem and log-Sobolev inequalities are
 established in \cite{O09} under positive Ricci curvature. See \cite{P16} for further extensions
 in this direction. Strict Kantorovich contractivity also plays an important role in the Monte Carlo method.
For example, it has been shown in \cite{JO10} that the constant $c$ in \eqref{kant-cont} provides upper bounds for the biases of empirical means when we use the Markov chain to simulate a given target distribution.
Moreover, strict Kantorovich contractivity is crucial in the study of the perturbation theory for Markov chains, and for obtaining quantitative bounds of the biases for Markov Chain
 Monte Carlo (MCMC)
algorithms \cite{PS14,JM17,RS18,MMS20,QH19}.

However,
verifying \eqref{kant-cont} is usually highly non-trivial in applications. In particular, requiring strict Kantorovich contractivity with respect to the underlying distance $d$ of the state space
can be too restrictive.
Instead, a natural approach is to modify the original distance involved in the strict Kantorovich contractivity condition \eqref{kant-cont}.
For this purpose, we need to take into consideration the following two
issues.
One is to find a good Markov coupling $((X,Y), \Pp_{x,y})$ for the transition kernels $p(x,dz)$ and $p(y,dz)$, and the second is to design a suitable distance-like function $\rho$ such that
\begin{equation}\label{rho-cont}
\Ee[\rho(X,Y)]\le (1-c_*)\rho(x,y),\quad x,y\in S
\end{equation}
holds for some $c_*\in (0,1)$. Here, $\rho(x,y)$ is called a distance-like function on $S\times S$, if $\rho(x,y)=0$ if and only if $x=y$, and $\rho(x,y)=\rho(y,x)>0$ for all $x\neq y\in S$.
Note that \eqref{rho-cont} naturally implies bounds in a related Kantorovich metric or semi-metric, cf.\ \cite[Lemma 2.1]{EM19}.
This approach has been widely used in the research on the ergodicity of SDEs driven by Brownian motions or pure jump L\'evy processes (see \cite{E11,E16, LW16,W16,M17,LW19}), and
it has proven to be very useful in the study of MCMC (see \cite{HSV14, BEZ20, BS20, EM19, MMS20}).

\smallskip

The present paper is strongly motivated by  \cite{EM19}, where two different kinds of distance-like functions $\rho$ for Markov chains on general metric state space were provided, and some quantitative bounds on contraction rates for \eqref{rho-cont} were also obtained. Here, on the one hand, we focus on the construction of
other
distance functions such that \eqref{rho-cont} is satisfied; on the other hand, we present explicit Markov couplings for general Markov chains on $\R^d$, whose transitions enjoy the following form
\begin{equation}\label{eq:chain}
x\mapsto x+hb(x)+ g(h)\xi, \quad x\in \R^d
\end{equation}
with $h, g(h)>0$, $b:\R^d\to \R^d$ and $\xi$ being a
``noise''
random variable,
with
an arbitrary
probability distribution $\mu$.
 Compared to \cite{EM19}, the improvements of our paper are as follows:
\begin{itemize}
	\item In \cite{EM19}, chains given by \eqref{eq:chain} were studied only with the Gaussian noise $\xi$. We cover a much larger class of (not necessarily isotropic)
	additive stochastic noises; including the $\alpha$-stable law with
	$\alpha\in (1,2)$.
	 We construct the corresponding Markov couplings such that \eqref{rho-cont} holds with appropriately designed distance-like functions $\rho(x,y)$.
	
\item The designed distance-like functions $\rho(x,y)$ in \eqref{rho-cont} and the associated lower bound estimates for the constant $c_*$ are more straightforward than in \cite{EM19},
which allows for a comparison of contraction rates for chains such as \eqref{eq:chain} driven by different types of noise. In particular, we show how contraction rates corresponding to the $\alpha$-stable noise can be larger than those corresponding to the Gaussian noise (cf.\ Remark \ref{remark:noises}). This is relevant in the context of applications of chains \eqref{eq:chain} in
MCMC methods. Indeed, recently there has been some interest in MCMC methods utilizing \eqref{eq:chain} with a heavy-tailed noise \cite{NguyenSimsekli2019, Simsekli2017, YeZhu}, based on the intuition that such chains may explore the state space better than their Gaussian counterparts, and hence the corresponding approximate sampling algorithms may exhibit faster convergence. Our observations aim to provide at least a partial rigorous justification for that intuition.
\item Besides contractions in terms of the total variation or the standard $L^1$-Wasserstein distance as in \cite{EM19}, we study contractions in terms of the $L^p$-Wasserstein distance (with $p>1$) as well.
Note that the $L^p$-Wasserstein distance $W_d^p$, by analogy to \eqref{eq:defWasserstein}, is defined by $W_d^p(\mu,\nu)=\left( \inf_{\xi\sim\mu,\eta\sim\nu}\Ee[d(\xi,\eta)^p] \right)^{1/p}$, and obtaining contractions with respect to $W_d^p$ requires using functions $\rho$ in \eqref{rho-cont} that are convex, rather than concave, at infinity.
\end{itemize}

We would like to point out that some of the contraction rates from \cite{EM19} for chains \eqref{eq:chain} with the Gaussian noise have been recently improved in \cite{DD19}, and related coupling methods have been also extended in \cite{DEEGM21} to study functional autoregressive processes (again, only with the Gaussian noise). However, in the present paper, as explained above, we focus on extending the results from \cite{EM19} in different directions.

The rest of the paper is arranged as follows. In Section \ref{section2}, we present general results for strict Kantorovich contractions of Markov chains on a separable metric space. In particular, the contractions in terms of the total variation, the $L^1$-Wasserstein distance and the $L^p$-Wasserstein distance with $p>1$ are studied. In Section \ref{section3}, we apply the results from Section \ref{section2} to Markov chains whose transitions involve additive stochastic noises. In particular, three different Markov couplings (via the reflection coupling and the refined basic coupling as well as their variants) are constructed in order to illustrate the practicality of the results in Section \ref{section2}.

\section{Strict Kantorovich contractions for Markov chains}\label{section2}
Let $p(x,dz)$ be a Markov transition kernel on a separable metric space $(S,d)$. Assume that $((X,Y), \Pp_{x,y})$ is a Markov coupling of $p(x,\cdot)$ and $p(y,\cdot)$ with $x,y\in S$. That is, there are random variables $X,Y:\Omega\rightarrow S$ defined on a common measurable space $(\Omega,\mathcal{F})$ and a probability kernel $(x,y,A)\mapsto \Pp_{x,y}(A)$ from $S\times S\times \mathscr{B}(S)$ to $[0,1]$ such that
$$
X\sim p(x,\cdot)\ \text{and}\ Y\sim p(y,\cdot)\quad \text{under } \Pp_{x,y}.
$$
Throughout this paper, we always assume that for all $x\in S$,
$$\Pp_{x,x}(d(X,Y)=0)=\Pp_{x,x}(X=Y)=1.$$
We denote by $\Ee_{x,y}$ the expectation with respect to $\Pp_{x,y}$.

For any fixed $x,y\in S$ and $l\ge0$, define
\begin{equation}\label{E:alpha}
\begin{split}
&\pi(x,y)=\Pp_{x,y}(d(X,Y)=0),\quad \beta(x,y)=\Ee_{x,y}[d(X,Y)-d(x,y)],\\
&\alpha_l(x,y)=\frac{1}{2}\Ee_{x,y}[(d(X,Y)-d(x,y))^2\I_{\{d(X,Y)<d(x,y)+l\}}].
\end{split}\end{equation}
For simplicity, we denote $\alpha(x,y)=\alpha_0(x,y)$.
Roughly speaking, $\pi(x,y)$ indicates the
probability
that the Markov coupling $((X,Y), \Pp_{x,y})$  will succeed
after one step; $\beta(x,y)$ is the drift of the Markov coupling, while $\alpha_l(x,y)$ reflects the fluctuations that mainly decrease the distance of the Markov coupling.
Note that the functions $\pi(x,y)$ and $\beta(x,y)$ have been used in \cite[(2.12) and (2.10)]{EM19} before, while $\alpha_l(x,y)$ here is a little different from (indeed is larger than)
the corresponding functions
in \cite[(2.13) and (2.22)]{EM19}.

The aim of this section is to construct various
distance-like
functions $\rho(x,y)$ such that the Markov coupling $((X,Y), \Pp_{x,y})$  is strictly contractive in the sense that for any $x,y\in S$,
\eqref{rho-cont}
holds with some $c_*\in (0,1)$.

To present the
distance-like
function $\rho(x,y)$ in an explicit way, we need the following notation. For any $r>0$ and $l\ge0$, set
\begin{equation}\label{E:alpha-}
\underline{\pi}(r)=\inf_{d(x,y)=r}\pi(x,y),\quad \overline{\beta}(r)=\sup_{d(x,y)=r}\beta(x,y),\quad   \underline{\alpha}_l(r)=\inf_{d(x,y)=r}\alpha_l(x,y).
\end{equation}
For simplicity we write $\underline{\alpha}(r)=\underline{\alpha}_0(r)$.

In the remaining part of this section, we will present four results on establishing \eqref{rho-cont} for several different distance-like functions $\rho$. In the first two instances (Theorems \ref{main-TV} and \ref{second-TV}), $\rho$ will be comparable with the total variation metric, in the third instance (Theorem \ref{thm-w1}) with the $L^1$-Wasserstein distance and in the final instance (Theorem \ref{thm-wq}) with the $L^p$-Wasserstein distance for $p > 1$. Hence the first three results can be considered as analogues of the results from \cite[Section 2]{EM19}, albeit with more straightforward formulas for $\rho$ and the corresponding contractivity constant $c_*$. This will be crucial in our analysis of contractions for Euler schemes in Section \ref{section3}. The fourth result has no counterpart in \cite{EM19}.

\subsection{Contraction in terms of the total variation}\label{total-variance}
In order to consider the contraction of Markov chains in terms of the total variation, we
require a positive probability that the Markov coupling
$((X,Y),\Pp_{x,y})$ will succeed
after one step when $d(x,y)$ is small; see (a1) in Assumption {\bf(A)}.

We first make the following assumption.

\noindent{\bf Assumption (A)} {\it There exist positive constants $c_0$ and $r_0\le r_1$ such that
\begin{itemize}
\item[{\rm(a1)}] $\inf_{r\in(0,r_0]}\underline{\pi}(r)>0${\rm;}

\item[{\rm(a2)}] $\inf_{r\in(r_0,r_1]}\underline{\alpha}(r)>0${\rm} and $\sup_{r\in(0,r_1]}\overline{\beta}(r)<\infty${\rm;}

\item[{\rm(a3)}] $\overline{\beta}(r)\le -c_0 r$ for all $r\in (r_1,\infty)$.
\end{itemize}}

Assumption {\bf(A)}{\rm(a3)} is a dissipative condition on the drift term for large distances. Such
conditions have been used to establish the exponential ergodicity of diffusions or SDEs with jumps; see \cite{E11,E16,LW16, W16, M17, LW19,LW20}.

Let
\begin{equation}\label{1f}
\rho(x,y)=a\I_{\{d(x,y)>0\}}+f_0(d(x,y)),
\end{equation}
where $f_0(r)= 1-\e^{-cr}+c\e^{-cr_1}r$,
and the constants $a$, $c > 0$ are chosen such that
\begin{equation}\label{1b}
c \ge \sup_{r\in(r_0,r_1]}\frac{4\overline{\beta}(r)_+}{\underline{\alpha}(r)}+1, \quad
a \ge 2c(1+e^{-c r_1})
 \sup_{r\in(0,r_0]}\frac{\overline{\beta}(r)_+}{\underline{\pi}(r)}+1
\end{equation}
with $\overline{\beta}(r)_+=\overline{\beta}(r) \vee 0$. By Assumption {\bf (A)}, the function $\rho$ is well defined, and it is easy to check that
$$
f_0'(r)=c\e^{-cr}+c\e^{-cr_1}>0,\quad  f_0''(r)=-c^2\e^{-cr}<0, \quad
f_0^{(3)}(r)=c^3\e^{-cr}>0.
$$ It also holds that there is a constant $\bar c\ge 1$ such that for all $x,y\in S$,
$$\bar c^{-1}\left(\I_{\{d(x,y)>0\}} +d(x,y)\right)\le \rho(x,y)\le \bar c \left(\I_{\{d(x,y)>0\}} +d(x,y)\right).$$

\begin{theorem}\label{main-TV}
Suppose that Assumption {\bf(A)} is satisfied. Let
$
\rho(x,y)
$
be defined by \eqref{1f}. Then there exists a constant $c_*\in (0,1)$ such that for all $x,y\in S$,
$$
\Ee_{x,y}[\rho(X,Y)]\le (1-c_*)\rho(x,y).
$$
\end{theorem}

\begin{proof} When $x=y$, the assertion holds trivially by our assumption on the Markov coupling $((X,Y),\Pp_{x,y})$ that $\Pp_{x,x}(d(X,Y)=0)=1$ and the fact that $\rho(x,x)=0$ for all $x\in S$.

Next, let $x,y\in S$ with $r:=d(x,y)>0$. Denote $R=d(X,Y)$. By the mean value theorem and the definition of $f_0$,
$$
\aligned
\rho(X,Y)-\rho(x,y)&=-a\I_{\{R=0\}}+f_0(R)-f_0(r)\\
&\le -a \I_{\{R=0\}}+f'_0(r)(R-r)+ \frac{1}{2}\left(\max_{\xi\in [r\wedge R, r\vee R]} f_0''(\xi)\right) (R-r)^2\\
&\le -a\I_{\{R=0\}}+f_0'(r)(R-r)+\frac{1}{2}f_0''(r)(R-r)^2\I_{\{R<r\}},
\endaligned
$$
where the last inequality follows from the facts that $f_0^{(3)}> 0$ and $f_0''< 0$ on $(0,\infty)$. By taking expectation and using the facts that $f_0'>0$ and $f''_0<0$ on $(0,\infty)$,  we get that
\begin{equation}\label{fR-fr}
\Ee_{x,y}[\rho(X,Y)]-\rho(x,y)\le -a\underline{\pi}(r)+\overline{\beta}(r)f_0'(r)+\underline{\alpha}(r)f_0''(r).
\end{equation}

(1) Suppose that $r\in(0,r_0]$. Then, it follows from (a1), (a2) and \eqref{1b} that
$$
\aligned
\Ee_{x,y}[\rho(X,Y)]-\rho(x,y)&\le -a\underline{\pi}(r)+(c\e^{-cr}+c\e^{-cr_1})\overline{\beta}(r)\le -a\underline{\pi}(r)+c(1+e^{-c r_1})\overline{\beta}(r)_+\\
&\le -\frac{a}{2}\underline{\pi}(r)\le -c_1\rho(x,y),
\endaligned
$$
where
$$
c_1=\frac{a}{2(a+1+cr_0\e^{-cr_1})}\inf_{r\in(0,r_0]}\underline{\pi}(r).
$$

(2) Suppose that $r\in(r_0,r_1]$. Then, by (a2), (a3) and \eqref{1b},
$$
\aligned
\Ee_{x,y}[\rho(X,Y)]-\rho(x,y)&\le \overline{\beta}(r)f'_0(r)+\underline{\alpha}(r)f''_0(r)=(c\e^{-cr}+c\e^{-cr_1})\overline{\beta}(r)-c^2\e^{-cr}\underline{\alpha}(r)\\
&\le 2c\e^{-cr}\overline{\beta}(r)_+-c^2\e^{-cr}\underline{\alpha}(r)\le-\frac{1}{2}c^2\e^{-cr_1}\underline{\alpha}(r)\le -c_2\rho(x,y),
\endaligned
$$
where
$$
c_2=\frac{c^2\e^{-cr_1}}{2(a+1+cr_1\e^{-cr_1})}\inf_{r\in(r_0,r_1]}\underline{\alpha}(r).
$$

(3) Suppose that $r\in(r_1,\infty)$. Then,
\begin{align*}
\Ee_{x,y}[\rho(X,Y)]-\rho(x,y)\le & (c\e^{-cr}+c\e^{-cr_1})\overline{\beta}(r)\le c\e^{-cr_1}\overline{\beta}(r)\\
\le & - cc_0\e^{-cr_1} r\le -c_3(a+1+c\e^{-cr_1}r)\le -c_3\rho(x,y),
\end{align*} where $$c_3=c_0[1+(1+a)r_1^{-1}c^{-1} e^{cr_1}]^{-1}.$$

Therefore, taking
$
c_*=\min\{c_1,c_2,c_3\}
$
gives us the desired assertion.
\end{proof}

\begin{remark}\label{remark:main-TV}
	In many applications, the constant $r_1$ in Assumption {\bf(A)}(a3) is uniquely determined as a parameter of the model under consideration. However, for many models, condition (a1) in Assumption {\bf(A)} can hold with any $r_0 \in (0,r_1]$, as long as the transition probability of the Markov chain has an unbounded support. Hence,
	in such cases
	one can easily simplify the statement of Theorem \ref{main-TV} by taking $r_0 = r_1$ and disregarding $\underline{\alpha}(r)$. Then the bounds for the constants $c$ and $a$ specified in \eqref{1b} become simplified and we can take
	\begin{equation*}
	c := 1, \quad
	a \ge 2(1+e^{-r_1}) \sup_{r\in(0,r_1]}\frac{\overline{\beta}(r)_+}{\underline{\pi}(r)}+1,
	\end{equation*}
	whereas the contractivity condition $\Ee_{x,y}[\rho(X,Y)]\le (1-c_*)\rho(x,y)$ in the statement of Theorem \ref{main-TV} holds with $c_* = \min \{ c_1, c_3 \}$, where
	\begin{equation*} c_1=\frac{a}{2(a+1+r_1\e^{-r_1})}\inf_{r\in(0,r_1]}\underline{\pi}(r), \quad c_3=c_0[1+(1+a)r_1^{-1} e^{r_1}]^{-1}.
	\end{equation*}
	We will use this simplified version of Theorem \ref{main-TV} in Section \ref{section3} in our analysis of contraction rates for Euler schemes with different types of noise.
\end{remark}

In the following, we consider
a variation
of Theorem \ref{main-TV},
in which
we will replace (a3) in Assumption {\bf(A)} by the following condition:
\begin{itemize}
\item[(a3$^*$)] \it There exist a measurable function $V:S\rightarrow [0,\infty)$, and constants $C_0\in(0,\infty)$ and $\lambda \in (0,1)$ such that
\begin{itemize}
\item[{\rm(i)}] for all $x\in S$, $$\int_S V(z)\,p(x,dz)\le (1-\lambda)V(x)+C_0;$$

\item[{\rm(ii)}] $\lim_{r\rightarrow\infty}\inf_{d(x,y)=r}[V(x)+V(y)]=\infty$ and $\lim_{r\rightarrow\infty}\sup_{d(x,y)=r}\frac{\beta(x,y)}{V(x)+V(y)}=0$.
\end{itemize}
\end{itemize}

Assumption (a3$^*$)(i) is a standard Lyapunov condition for the exponential ergodicity of the Markov chain with transition kernel $p(x,dz)$; see \cite{MT}.
For any fixed $K>0$, we define
\begin{equation}\label{eq:r1}
r_1=\sup\left\{d(x,y):\ x,y\in S \text{ with }  \frac{\beta(x,y)}{V(x)+V(y)}\ge K\  \text{or } V(x)+V(y)\le \frac{4C_0}{\lambda}\right\}.
\end{equation}
In particular,  by (ii) in condition (a3$^{*}$), $r_1<\infty.$
Under conditions (a1), (a2) and (a3$^*$), we  are concerned
with the following distance-like
function
\begin{equation}\label{2f--}
\rho(x,y)=a\I_{\{d(x,y)>0\}}+ f_1(d(x,y))+\epsilon\left(V(x)+V(y)\right)\I_{\{d(x,y)>0\}},
\end{equation}
where  $
f_1(r)=1-\e^{-cr}
$, $$a=2\left(c\sup_{r\in(0,r_0]}\overline{\beta}(r)_++2\epsilon C_0\right)\left[\inf_{r\in(0,r_0]}\underline{\pi}(r)\right]^{-1}, \quad \epsilon= \frac{1}{8C_0}c^2\e^{-cr_1}\inf_{r\in(r_0, r_1]}\underline{\alpha}(r),
$$ and
$$c=\sup_{r\in(r_0,r_1]}\frac{2\overline{\beta}(r)_+}{\underline{\alpha}(r)}+A+1, \quad A=\frac{16KC_0}{\lambda\inf_{r\in(r_0,r_1]}\underline{\alpha}(r)}.$$
Consequently, there is a constant $\bar c\ge1$ such that for all $x,y\in S$,
$$\bar c^{-1} \I_{\{d(x,y)>0\}} (1+V(x)+V(y))\le \rho(x,y)\le \bar c \I_{\{d(x,y)>0\}} (1+V(x)+V(y)).$$ Thus, using the distance-like
function $\rho(x,y)$ defined by \eqref{2f--}, we can consider the convergence to equilibrium in terms of the  weighted total variation metric; see \cite{HM} for related discussions on this topic.

\begin{theorem}\label{second-TV}
Suppose that {\rm(a1)} and {\rm(a2)} in Assumption {\bf (A)} and {\rm(a3$^*$)} hold. Let
$
\rho(x,y)
$
be  defined by \eqref{2f--}.  Then there exists a constant $c_*\in(0,1)$ such that for all $x,y \in S$,
$$
\Ee_{x,y}[\rho(X,Y)]\le (1-c_*)\rho(x,y).
$$
\end{theorem}

\begin{proof}
Let $x,y\in S$ with $r=d(x,y)>0$. According to the argument for \eqref{fR-fr}
and (i) in condition (a3$^*$),
\begin{equation}\label{rho}
\Ee_{x,y}[\rho(X,Y)]-\rho(x,y)\le -a\underline{\pi}(r)+f_1'(r)\overline{\beta}(r)+f_1''(r)\underline{\alpha}(r)-\lambda\epsilon(V(x)+V(y))+2\epsilon C_0.
\end{equation}

(1) When $r\in(0,r_0]$, it follows from \eqref{rho} and the definition of the constant $a$
that
$$
\aligned
\Ee_{x,y}[\rho(X,Y)]-\rho(x,y)&\le-a\underline{\pi}(r)+c\e^{-cr}\overline{\beta}(r)-\lambda\epsilon(V(x)+V(y))+2\epsilon C_0\\
&\le-a\underline{\pi}(r)+c\overline{\beta}(r)_+-\lambda\epsilon(V(x)+V(y))+2\epsilon C_0\\
&\le -\frac{a}{2}\underline{\pi}(r)-\lambda\epsilon(V(x)+V(y))\\
&\le -c_1\rho(x,y),
\endaligned
$$
where
$$
c_1=\min\left\{\frac{a}{2(a+1)}\inf_{r\in(0,r_0]}\underline{\pi}(r),\ \lambda\right\}.
$$

(2) When $r\in(r_0,r_1]$, we get from \eqref{rho} and the definitions of the constants $c$
and $\epsilon$
that
$$
\aligned
\Ee_{x,y} [\rho(X,Y)]-\rho(x,y) &\le c\e^{-cr}\overline{\beta}(r)- c^2\e^{-cr}\underline{\alpha}(r)-\lambda\epsilon(V(x)+V(y))+2\epsilon C_0\\
&\le -\frac{1}{2}c^2\e^{-cr}\underline{\alpha}(r)-\lambda\epsilon(V(x)+V(y))+2\epsilon C_0\\
&\le -\frac{1}{4}c^2\e^{-cr}\underline{\alpha}(r)-\lambda\epsilon(V(x)+V(y))\\
&\le  -\frac{1}{4}c^2\e^{-cr_1}\inf_{r\in(r_0,r_1]}\underline{\alpha}(r)-\lambda\epsilon(V(x)+V(y))\\
&\le -c_2\rho(x,y),
\endaligned
$$
where
$$
c_2=\min\left\{ \frac{c^2\e^{-cr_1}\inf_{r\in(r_0,r_1]}\underline{\alpha}(r)}{4(a+1)}, \lambda\right\}.
$$

(3) When $r\in (r_1,\infty)$, by \eqref{rho} and the definition of $r_1$, we get that
$$
\aligned
\Ee_{x,y}[\rho(X,Y)]-\rho(x,y)&\le f_1'(r)\overline{\beta}(r)-\lambda\epsilon(V(x)+V(y))+2\epsilon C_0\\
&\le \left(Kc\e^{-cr_1}-\frac{\epsilon\lambda}{2}\right)(V(x)+V(y))\le -c_3 (V(x)+V(y)),
\endaligned$$
where $$c_3:=\frac{\lambda}{16C_0}c\e^{-cr_1}\Big(\inf_{r\in(r_0,r_1]}\underline{\alpha}(r)\Big)\Big[\sup_{r\in(r_0,r_1]}\frac{2\overline{\beta}(r)_+}{\underline{\alpha}(r)}+1\Big]$$ due to the definitions of constants $A$ and $c$.
This along with the definition of $r_1$ again further yields that
$$
\Ee_{x,y}[\rho(X,Y)]-\rho(x,y)\le -\frac{2C_0c_3}{\lambda}-\frac{c_3}{2}(V(x)+V(y))
\le -c_4\rho(x,y),
$$
where
$$
c_4=\min\left\{\frac{2C_0c_3}{\lambda(a+1)},\  \frac{c_3}{2\epsilon} \right\}.
$$

Combining with all the estimates above, we can obtain the desired assertion with $c_*=\min\{c_1,c_2,c_4\}$.
\end{proof}

\begin{remark}
Similarly to Remark \ref{remark:main-TV} that follows
 Theorem \ref{main-TV}, we would like to point out that in many applications $r_0$ in condition {\rm(a1)} in Assumption {\bf (A)} can be chosen arbitrarily. Hence, if we choose $r_0 = r_1$, where $r_1$ is given by \eqref{eq:r1}, we can simplify the statement of Theorem \ref{second-TV}. Then we have
\begin{equation*}
\epsilon= \frac{1}{8C_0}c^2\e^{-cr_1}, \quad
c=A+1, \quad A=\frac{16KC_0}{\lambda},
\end{equation*}	
whereas $a$ is given by the same formula as above, with $r_0 = r_1$. Hence the constant $c_*$ in the statement of Theorem \ref{second-TV} becomes $c_* = \min \{ c_1, c_4 \}$ with the formulas for $c_1$ and $c_4$ unchanged, but with $c_3 = \frac{\lambda}{16C_0}c\e^{-cr_1}$.
\end{remark}

\subsection{Contraction in terms of the Wasserstein distance}\label{Wasserstein}
  When the Markov coupling $((X,Y),\Pp_{x,y})$ starting from $(x,y)$ can not succeed
  after one step, one may not expect the contraction of $((X,Y),\Pp_{x,y})$ in terms of the total variation. Instead, in this case we will study the contraction in terms of the Wasserstein-type distance, whose associated distance function $\rho(x,y)$ satisfies that $\lim_{d(x,y)\to0} \rho(x,y)=0.$

  Note that the first part of this subsection is similar in spirit to
  \cite[Subsection 2.2]{EM19}. However, in the second part we will expand our approach to cover contractions in the $L^q$-Wasserstein distances for $q > 1$, which have not been considered in \cite{EM19}.

Recall that $\alpha_l(x,y)$ and $\alpha_l(r)$ are defined by \eqref{E:alpha} and \eqref{E:alpha-}, respectively.
We suppose that the following assumption holds:

\noindent{\bf Assumption (B)}: {\it There are a nonnegative $C([0,\infty))\cap C^2((0,\infty))$ function $\Psi$ such that $\Psi(0)=0$, $\Psi'>0$,
$\Psi''\le 0$ and $\Psi''$ is non-increasing	
on $(0,\infty)$, and constants $l_0, r_1, c_0>0$ such that
\begin{itemize}
\item[{\rm(b1)}] $\inf_{r\in(0,r_1]}\frac{\underline{\alpha}_{l_0}(r)}{r}>0${\rm;}

\item[{\rm (b2)}]  $\Big[\sup_{r\in(0,r_1]}\frac{2\overline{\beta}(r)_+}{\Psi'(r+l_0)\underline{\alpha}_{l_0}(r)}+1\Big]\Psi(l_0)\le \log 2${\rm ;}

\item[{\rm (b3)}] $\overline{\beta}(r)\le -c_0r$ for all $r\in (r_1,\infty)$.

\end{itemize}

\rm

Note that we use $\underline{\alpha}_l(r)$ with $l>0$ in Assumption {\bf (B)} rather than $\underline{\alpha}(r):=\underline{\alpha}_0(r)$. Indeed, by the definition of $\underline{\alpha}(r)$, we have $\underline{\alpha}(r)\le r^2/2$ and so (b1) is not satisfied. That is, using $\alpha_l(r)$ with
$l>0$ is crucial in this subsection. Define \begin{equation}\label{2f}
f_3(r)=
\begin{cases}
\int_0^r\e^{-c\Psi(s)}\,ds,\quad &0\le r\le r_1+l_0, \\
f_3(r_1+l_0)+\frac{f_3'(r_1+l_0)}{2}\int_{r_1+l_0}^r\big[1+\exp\big(\frac{2f_3''(r_1+l_0)}{f_3'(r_1+l_0)}(s-(r_1+l_0))\big)\big]\,ds,\quad &r>r_1+l_0,
\end{cases}
\end{equation} where $$
c=\sup_{r\in(0,r_1]}\frac{2\overline{\beta}(r)_+}{\Psi'(r+l_0)\underline{\alpha}_{l_0}(r)}+1.
$$ In particular, there is a constant $\bar c\ge1$ such that for all $r\ge0$,
$$\bar c^{-1}r\le f_3(r)\le \bar cr.$$

\begin{theorem}\label{thm-w1}
Suppose that Assumption {\bf(B)} is satisfied. Let
$
\rho(x,y)=f_3(d(x,y)),
$
where $f_3$ is defined by \eqref{2f}.  Then there exists a constant $c_*\in (0,1)$ such that for all $x,y \in S$,
$$
\Ee_{x,y}[\rho(X,Y)]\le (1-c_*)\rho(x,y).
$$
\end{theorem}
\begin{proof}
It is clear that $f_3\in C([0,\infty))\cap C^2((0,\infty))$ such that
$$
f_3'(r)=
\begin{cases}
\e^{-c\Psi(r)},\quad &0< r\le r_1+l_0,\\
\frac{f_3'(r_1+l_0)}{2}\big[1+\exp\big(\frac{2f_3''(r_1+l_0)}{f_3'(r_1+l_0)}(r-(r_1+l_0))\big)\big],\quad &r>r_1+l_0,
\end{cases}
$$ and
$$
f_3''(r)=
\begin{cases}
-c\Psi'(r)\e^{-c\Psi(r)},\quad &0< r\le r_1+l_0,\\
f_3''(r_1+l_0)\exp\big(\frac{2f_3''(r_1+l_0)}{f_3'(r_1+l_0)}(r-(r_1+l_0))\big),\quad &r>r_1+l_0.
\end{cases}
$$ In particular, for all $r\ge 0$,
$$\frac{f_3'(r_1+l_0)r}{2}\le f_3(r)\le \max\left\{1, \frac{f_3(r_1+l_0)}{r_1+l_0}\right\}r,$$ and
$f_3'>0$, $f_3''\le 0$ and $f_3''$ is increasing on $(0,\infty)$.

Let $x,y\in S$ with $r=d(x,y)>0$. If $r\in (0,r_1]$, then, by the definitions of $f_3$, $c$ and (b1)--(b2),
$$
\aligned
\Ee_{x,y}[f_3(R)-f_3(r)]&\le f_3'(r)\overline{\beta}(r)+ f_3''(r+l_0)\underline{\alpha}_{l_0}(r)\\
&=\e^{-c\Psi(r)}[\overline{\beta}(r)- c\Psi'(r+l_0)\e^{-c(\Psi(r+l_0)-\Psi(r))}\underline{\alpha}_{l_0}(r)]\\
&\le \e^{-c\Psi(r)}[\overline{\beta}(r)
-2\e^{-c(\Psi(r+l_0)-\Psi(r))}\overline{\beta}(r)_+
-\Psi'(r+l_0)\e^{-c(\Psi(r+l_0)-\Psi(r))}\underline{\alpha}_{l_0}(r)]\\
&\le -\Psi'(r+l_0)\e^{-c\Psi(r+l_0)}\frac{\underline{\alpha}_{l_0}(r)}{r}r\le
-c_1f_3(r),
\endaligned
$$
where the third
 inequality holds true since $\sup_{r\in (0,r_1]}(\Psi(r+l_0)-\Psi(r))\le \Psi(l_0)$, due to
 the fact that $\Psi(0)=0$
 and
 $\Psi''\le0$,
and hence, by (b2) and the definition of $c$,
$$
\overline{\beta}(r)
-2\e^{-c(\Psi(r+l_0)-\Psi(r))}\overline{\beta}(r)_+
\le \overline{\beta}(r)
-2\e^{-c\Psi(l_0)}\overline{\beta}(r)_+
\le \overline{\beta}(r)-\overline{\beta}(r)_+\le 0,
$$
whereas in the last inequality  $$
c_1=\Psi'(r_1+l_0)e^{-c\Psi(r_1+l_0)}\inf_{r\in(0,r_1]}\frac{\underline{\alpha}_{l_0}(r)}{r}>0,
$$
thanks to (b1) and $\Psi'>0$ on $(0,\infty)$.

If $r\in (r_1,\infty)$, then, by (b3),
$$\Ee_{x,y}[f_3(R)-f_3(r)]\le f_3'(r)\overline{\beta}(r)\le -c_0 f_3'(r_1+l_0)r/2\le -c_2f_3(r),$$
where $$c_2:=\frac{c_0 }{2}\e^{-c(r_1+l_0)}\max\left\{1, \frac{f_3(r_1+l_0)}{r_1+l_0}\right\}^{-1}.$$

The proof is completed by taking $c_*=c_1\wedge c_2.$
\end{proof}

\begin{remark}\rm
Take $\Psi(r)=r$, then (b2) in Assumption {\bf(B)}  becomes
\begin{itemize}
\item[{\rm(b2$^*$)}] $\Big[\sup_{r\in(0,r_1]}\frac{2\overline{\beta}(r)_+}{\underline{\alpha}_{l_0}(r)}+1\Big]l_0\le \log 2$.
\end{itemize} With this special choice,  (b1) and (b2$^*$) are
equivalent to
\cite[(B1) and (B2) in Subsection 2.2]{EM19}.
\end{remark}

\rm

In the following, we will consider the contraction in terms of the $L^q$-Wasserstein distance with $q>1$. For this, we assume that there is a constant  $l>0$ such that the Markov coupling $((X,Y),\Pp_{x,y})$ satisfies
 $d(X,Y)\le d(x,y)+l$ for all $x$, $y \in S$.
 An example of such coupling will be used in Section \ref{section3}, see \eqref{basic-reflec1} and \eqref{basic-reflec2} therein.
Let $\Psi$, $l_0$, $r_1$ and $c_0$ be given in Assumption {\bf(B)}. We now define \begin{equation}\label{4fff}f_4(r)=
\begin{cases}
\int_0^r\e^{-c\Psi(s)}\,ds,\quad &0\le r\le r_1+l_0, \\
f_4(r_1+l_0)+A(r-(r_1+l_0))^p+ \frac{1}{c}\e^{-c\Psi(r_1+l_0)}(1-\e^{-c(r-(r_1+l_0))}),\quad &r>r_1+l_0,
\end{cases}\end{equation} where $$
c=\sup_{r\in(0,r_1]}\frac{2\overline{\beta}(r)_+}{\Psi'(r+l_0)\underline{\alpha}_{l_0}(r)}+1,\quad  p>2, \quad A=(p(p-1))^{-1}(k(r_1+l_0))^{2-p}c e^{-c(\Psi(r_1+l_0)+k(r_1+l_0))}$$ and $$r_1\ge l+1,\quad k\ge 1+\max\left\{2le^{cl}/(c_0(r_1+l_0)),4^pl/(c_0(r_1+l_0))\right\}.$$
We can check that $f_4\in C([0,\infty))\cap C^2((0,\infty))$ such that $f_4'>0$ on $(0,\infty)$, $f_4''<0$ on $(0,(k+1)(r_1+l_0))$, $f_4''((k+1)(r_1+l_0))=0$ and $f_4''$ is increasing on $(0,\infty)$. Moreover, there is a constant $\bar c\ge1$ such that for all $r>0$,
$$\bar c^{-1}(r\vee r^p)\le f_4(r)\le \bar c (r\vee r^p).$$ In particular, for any $\theta\in [1,p]$, there is a constant $c_\theta>0$ such that
for all $r>0$,
$$c_\theta r^\theta\le f_4(r).$$ Thus, with the choice of $f_4(r)$ defined by \eqref{4fff}, one may consider the contraction of the Markov coupling $((X,Y),\Pp_{x,y})$ in the $L^q$-Wasserstein distance with $q\in [1,p]$.

\begin{theorem}\label{thm-wq}
Suppose that Assumption {\bf(B)} is satisfied, and there is a constant  $l>0$
such
that the Markov coupling $((X,Y),\Pp_{x,y})$ satisfies
$d(X,Y)\le d(x,y)+l$ for all $x$, $y \in S$. Let $\rho(x,y)=f_4(d(x,y))$, where $f_4$ is defined by \eqref{4fff}. Then there exists a constant $c_*\in (0,1)$ such that for all $x,y \in S$,
$$
\Ee_{x,y}[\rho(X,Y)]\le (1-c_*)\rho(x,y).
$$
\end{theorem}

\begin{proof}
Let $x,y\in S$ with $r=d(x,y)>0$. Suppose that $r\in (0,r_1]$. Then, by using
the fact that $f_4''$ is increasing,
and following the arguments
in the proof of
Theorem \ref{thm-w1}, we can obtain that
$$
\aligned
\Ee_{x,y}[f_4(R)-f_4(r)]
&\le f_4'(r)\overline{\beta}(r)+ f_4''(r+l_0)\underline{\alpha}_{l_0}(r)\\
&=\e^{-c\Psi(r)}[\overline{\beta}(r)- c\Psi'(r+l_0)\e^{-c(\Psi(r+l_0)-\Psi(r))}\underline{\alpha}_{l_0}(r)]\\
&\le -c_1f_4(r),
\endaligned
$$ where $$
c_1=\Psi'(r_1+l_0)e^{-c\Psi(r_1+l_0)}\inf_{r\in(0,r_1]}\frac{\underline{\alpha}_{l_0}(r)}{r}>0.
$$

Suppose that $r\in(r_1,(k+1)(r_1+l_0)-l]$. By (b3) and $f_4''<0$ on $(0,(k+1)(r_1+l_0)]$, we get
$$
\Ee_{x,y}[f_4(R)-f_4(r)]
\le f_4'(r)\overline{\beta}(r)
\le -c_0f_4'(r)r\le -c_2f_4(r)
$$
for some $c_2\in(0,1)$.

Now consider the case
where
$r\in ((k+1)(r_1+l_0)-l, 4k(r_1+l_0)]$. It follows from the properties of the Markov coupling $((X,Y),\Pp_{x,y})$ and the function $f_4$ that
$$
\aligned
f_4(R)-f_4(r)&=(f_4(R)-f_4(r))\I_{\{R\le r\}}+(f_4(R)-f_4(r))\I_{\{r<R\le r+l\}}\\
&\le \left[\inf_{\xi\in(R,r)}f_4'(\xi)\right](R-r)\I_{\{R\le r\}}+\left[\sup_{\xi\in(r,R)}f_4'(\xi)\right](R-r)\I_{\{r<R\le r+l\}}\\
&\le f_4'((k+1)(r_1+l_0))(R-r)\I_{\{R\le r\}}+[f_4'(r)\vee f_4'(r+l)](R-r)\I_{\{r<R\le r+l\}}\\
&\le f_4'((k+1)(r_1+l_0))(R-r)\\
&\quad+\left[f_4'(r)\vee f_4'(r+l)-f_4'((k+1)(r_1+l_0))\right](R-r)\I_{\{r<R\le r+l\}}.
\endaligned
$$
Taking expectation and (b3) give us that
$$
\aligned
\Ee_{x,y}[f_4(R)-f_4(r)]&\le -c_0f_4'((k+1)(r_1+l_0))r+l[f_4'(r)\vee f_4'(r+l)-f_4'((k+1)(r_1+l_0))]\\
&\le -\frac{c_0}{2}f_4'((k+1)(r_1+l_0))r\le -c_3f_4(r)
\endaligned
$$
for some $c_3\in(0,1)$. Here we used the fact that for all $r\in ((k+1)(r_1+l_0)-l, 4k(r_1+l_0)]$,
\begin{equation}\label{kr1-4kr1}
\frac{c_0}{2}f_4'((k+1)(r_1+l_0))r\ge l(f_4'(r)\vee f_4'(r+l)-f_4'((k+1)(r_1+l_0))).
\end{equation}
Indeed, by the definition of $f_4$, \eqref{kr1-4kr1} is a consequence of $$
 \frac{c_0}{2}\e^{-c(k+1)(r_1+l_0)}r\ge le^{-cr},\quad \frac{c_0}{2}(k(r_1+l_0))^{p-1}r\ge  l(r+l-(r_1+l_0))^{p-1}
$$ for all $r\in ((k+1)(r_1+l_0)-l, 4k(r_1+l_0)]$, thanks to the fact that $k\ge \max\{2l\e^{cl}/(c_0(r_1+l_0)),4^pl/(c_0(r_1+l_0))\}$.

Finally, suppose that $r\in (4k(r_1+l_0),\infty)$. Since $f_4''>0$ on $((k+1)(r_1+l_0),\infty)$,
\begin{align*}
f_4(R)-f_4(r)&=(f_4(R)-f_4(r))\I_{\{R\le r/2\}}+(f_4(R)-f_4(r))\I_{\{ r/2<R\le r\}}\\
&\quad +(f_4(R)-f_4(r))\I_{\{r<R\le r+l\}}\\
&\le (f_4(r/2)-f_4(r))\I_{\{R\le r/2\}}+\left[\inf_{\xi\in(R,r)}
f_4'(\xi)\right](R-r)\I_{\{ r/2<R\le r\}}\\
&\quad+\left[\sup_{\xi\in(r,R)}f'_4(\xi)\right](R-r)\I_{\{r<R\le r+l\}}\\
&\le \left[\inf_{\xi\in(r/2,r)}f'_4(\xi)\right](r/2-r)\I_{\{R\le r/2\}}+f'_4(r/2)(R-r)\I_{\{r/2<R\le r\}}\\
&\quad +f'_4(r+l)(R-r)\I_{\{r<R\le r+l\}}\\
&\le f'_4(r/2)(r/2-r)\I_{\{R\le r/2\}}+f'_4(r/2)(R-r)\I_{\{ r/2<R\le r\}}\\
&\quad +f'_4(r+l)(R-r)\I_{\{r<R\le r+l\}}\\
&\le \frac{1}{2}f'_4(r/2)(R-r)\I_{\{R\le r\}}+f'_4(r+l)(R-r)\I_{\{r<R\le r+l\}}\\
&\le \frac{1}{2}f'_4(r/2)(R-r)+\Big(f'_4(r+l)-\frac{1}{2}f'_4(r/2)\Big)(R-r)\I_{\{r<R\le r+l\}}.
\end{align*}
Thus, by (b3) we have
$$
\Ee_{x,y}[f_4(R)-f_4(r)]\le -\frac{c_0}{2}f'_4(r/2)r+l\left(f'_4(r+l)-\frac{1}{2}f'_4(r/2)\right)
\le -\frac{c_0}{4}f'_4(r/2)r\le -c_4f_4(r)
$$
for some $c_4\in(0,1)$. Here we used the fact that
$$
\frac{c_0}{4}f'_4(r/2)r\ge l\left(f'_4(r+l)-\frac{1}{2}f'_4(r/2)\right),\quad r> 4k(r_1+l_0),
$$ which follows from
$$
\frac{c_0}{4}\e^{-c r/2 }r\ge l \e^{-cr},\quad \frac{c_0}{4}(r/2-(r_1+l_0))^{p-1}r> lr^{p-1}  \quad \text{for all } r>4k (r_1+l_0)
$$ and the fact that $k\ge \max\{2l\e^l/(c_0(r_1+l_0)), l4^p/(c_0(r_1+l_0))\}.$

Combining all the estimates above, we prove the desired assertion.
\end{proof}

\section{Applications: Euler Schemes of SDEs}\label{section3}
There are
numerous
works devoted to Euler discretizations of the following stochastic differential equation (SDE)
\begin{equation}\label{e:SDE1}
dX_t=b(X_t)\,dt+dZ_t,
\end{equation}
where $b:\rd\rightarrow\rd$ is Lipschitz continuous, and $(Z_t)_{t\ge0}$ is a stochastic process on $\rd$. In particular, when $(Z_t)_{t\ge0}$ is a rotationally invariant symmetric $\alpha$-stable process with $\alpha \in (0,2]$ (when $\alpha=2$, $(Z_t)_{t\ge0}$ is a Brownian motion), the transitions of the Markov chain for the Euler schemes corresponding to the SDE \eqref{e:SDE1} with step size $h>0$ are given by
\begin{equation}\label{e:SDE2}
x\mapsto x+hb(x)+ h^{1/\alpha}\xi, \quad x\in \R^d,
\end{equation} where $\xi$ is a random variable with $\alpha$-stable law (when $\alpha=2$, $\xi$ is a random variable with normal distribution).
The purpose of this section is to apply
the
results in the previous section to study the strict contraction of Markov chains including the system \eqref{e:SDE2}.

Let $h>0$, $g$ be a continuous and strictly increasing function on $[0,\infty)$ with $g(0)=0$, and  $\xi$ be a random variable whose distribution is given by $\mu$. We will consider the Markov chain $X$ on $\R^d$ whose transitions are of the following form:
\begin{equation}\label{eq:chain2}
x\mapsto x+hb(x)+g(h)\xi,\quad x\in \R^d.
\end{equation}
Throughout this section, we always assume that the coefficient $b(x)$ satisfies the following assumption.

\noindent{\bf Assumption (C):}{\it
\begin{itemize}
\item[{\rm (c1)}] There is a constant $L\ge0$ such that $|b(x)-b(y)|\le L|x-y|$ for $x,y\in\rd${\rm;}

\item[{\rm (c2)}] There are constants $K>0$ and $\mathcal R\ge0$ such that $\langle x-y, b(x)-b(y)\rangle\le -K|x-y|^2$ for $x,y\in \R^d$ with $|x-y|\ge\mathcal{R}$.

\end{itemize}}
Note that it must hold that $K\le L$. We will construct three explicit Markov couplings of the chain $X$ according to different conditions on the distribution $\mu$ of the random variable $\xi$.

\subsection{General case}\label{app-add}
For any $x\in \R^d$ and $\kappa>0$, set
\begin{equation}\label{mu_x}
(x)_\kappa:=\Big(1\wedge\frac{\kappa}{|x|}\Big)x,\quad \mu_{x}(dz):=\big(\mu\wedge(\delta_x\ast \mu)\big)(dz).
\end{equation} In this part, we suppose that the following condition holds for the measure $\mu$.
\begin{itemize}\it
\item[{\rm (c3)}] There is a constant $\kappa_0>0$ such that \begin{equation}\label{kappa}
J_{\kappa_0}:=\inf_{|z|\le \kappa_0}\big(\mu\wedge(\delta_{z}\ast \mu)\big)(\rd)>0.
\end{equation}
\end{itemize}}

Write $\hat x=\hat x_h:=x+hb(x)$. We will adopt the following Markov coupling $((X,Y),\Pp_{x,y})$ of the chain $X$:
\begin{equation}\label{ref-basic}
\begin{cases}
X=\,\,\hat{x}+g(h)z, \qquad\qquad\qquad \qquad\qquad\,\,\,\, \mu(dz);\\
Y=
\begin{cases}
\hat{y}+g(h)(z+g(h)^{-1}(\hat{x}-\hat{y})_\kappa),&\quad \frac{1}{2}\mu_{g(h)^{-1}(\hat{y}-\hat{x})_\kappa}(dz),  \\
\hat{y}+g(h)(z+g(h)^{-1}(\hat{y}-\hat{x})_\kappa),&\quad  \frac{1}{2}\mu_{g(h)^{-1}(\hat{x}-\hat{y})_\kappa}(dz),\\
\hat{y}+g(h)z,&\quad \mu(dz)-\frac{1}{2}\mu_{g(h)^{-1}(\hat{y}-\hat{x})_\kappa}(dz)-\frac{1}{2}\mu_{g(h)^{-1}(\hat{x}-\hat{y})_\kappa}(dz).
\end{cases}\end{cases}
\end{equation} To check that $((X,Y),\Pp_{x,y})$ is a Markov coupling, we apply the fact that
$(\delta_{-x}*\mu_{x})(dz)=\mu_{-x}(dz)$ for all $ x\in \R^d$ (see Lemma \ref{coup-1} in the appendix of this paper).
 In particular, this implies that \begin{equation}\label{e:mea}\mu_{x}(\R^d)=\mu_{-x}(\R^d).\end{equation} Indeed, the coupling above is motivated by the refined basic coupling for SDEs with additive L\'evy noises introduced in \cite{LW19}.
  Here we
 adapt it to Markov chains of the form \eqref{eq:chain2}.
Note that the parameter $\kappa$ in \eqref{mu_x} is introduced so that the coupling \eqref{ref-basic} has a positive probability of being successful even when the jump distribution $\mu$ has finite support. For distribution
 $\mu$ with full support, the value of $\kappa$ can be chosen arbitrarily, in particular taking $\kappa = \infty$ gives $(x)_{\kappa} = x$. An extended discussion on the construction of \eqref{ref-basic} can be found in \cite[Section 2.1]{LW19}.

Recall that, for $x,y\in\R^d$,
$$
\beta(x,y)=\Ee_{x,y}[R-r],\quad \alpha(x,y)=\frac{1}{2}\Ee[(R-r)^2\I_{\{R< r\}}],\quad \pi(x,y)=\Pp_{x,y}(R=0),
$$ where $r=|x-y|$ and $R=|X-Y|$. We further set $\hat{r}=|\hat{x}-\hat{y}|$.

\begin{lemma}\label{beta}
Under Assumption {\bf(C)} and condition {\rm(c3)}, we have
\begin{itemize}
\item[{\rm(i)}]
$\beta(x,y)\le hL |x-y|$ for any $x,y\in\rd${\rm;}

\item[{\rm(ii)}]
$\beta(x,y)\le -(K-hL^2/2)h|x-y|$ for any $x,y\in \R^d$ with $|x-y|\ge \mathcal{R}$ and any $h\le 2KL^{-2}${\rm;}

\item[\rm{(iii)}] $\pi(x,y)\ge \frac{1}{2}J_{g(h)^{-1}\kappa}\I_{\{\hat{r}\le \kappa\}}$ for any $x,y\in \R^d$ and $0<\kappa\le g(h)\kappa_0$. In particular, for any $x,y\in \R^d$ with $|x-y|\le \kappa/(1+hL)$ and $0<\kappa\le g(h)\kappa_0$,
    $$\pi(x,y)\ge \frac{1}{2}J_{g(h)^{-1}\kappa}>0 ;$$

\item[\rm{(iv)}]$
\alpha(x,y)\ge \frac{1}{4}J_{g(h)^{-1}\kappa}\big((\hat{r}-\hat{r}\wedge\kappa)-r\big)^2\I_{\{\hat{r}-\hat{r}\wedge\kappa<r\}}
$ for any $x,y\in \R^d$ and $0<\kappa\le g(h)\kappa_0$.
In particular, for any $x,y\in \R^d$ with $|x-y|\le \mathcal{R}$, $0<\kappa\le g(h)\kappa_0$ and any $0<h\le \kappa(2L\mathcal{R})^{-1}$,
$$
\alpha(x,y)\ge\bigg(\frac{r^2}{4}\wedge\frac{\kappa^2}{16}\bigg) J_{g(h)^{-1}\kappa}>0.
$$

\end{itemize}
\end{lemma}

\begin{proof} Let
$\hat{\beta}(x,y) =\Ee_{x,y}[R-\hat{r}]$.
We first claim that for all $x,y\in\rd$, $\hat{\beta}(x,y)=0$.
Indeed, for fixed $x,y\in\rd$ with $r=|x-y|>0$, by \eqref{ref-basic} and \eqref{e:mea},  we get
$$
\aligned
\Ee_{x,y}[R]&=\frac{1}{2}\int_{\rd}(\hat{r}-\hat{r}\wedge \kappa)\,\mu_{g(h)^{-1}(\hat{y}-\hat{x})_\kappa}(dz)
 +\frac{1}{2}\int_{\rd}(\hat{r}+\hat{r}\wedge \kappa)\,\mu_{g(h)^{-1}(\hat{x}-\hat{y})_\kappa}(dz)\\
&\quad+\int_{\rd} \hat{r}\,\left[\mu(dz)-\frac{1}{2}\mu_{g(h)^{-1}(\hat{y}-\hat{x})_\kappa}(dz)-\frac{1}{2}\mu_{g(h)^{-1}(\hat{x}-\hat{y})_\kappa}(dz)\right]\\
&=\hat{r}+\frac{1}{2}(\hat{r}\wedge \kappa)(\mu_{g(h)^{-1}(\hat{x}-\hat{y})_\kappa}(\rd)-\mu_{g(h)^{-1}(\hat{y}-\hat{x})_\kappa}(\rd))=\hat{r},
\endaligned
$$
which implies that $\hat{\beta}(x,y)=0$ for all $x,y\in\rd$.
Hence, for any $x,y\in \R^d$, $$\beta(x,y)=\hat{\beta}(x,y)+\hat{r}-r=\hat{r}-r.$$ This along with Assumption (c1) yields that for any $x,y\in \R^d$,
$$
\beta(x,y)\le |\hat{r}-r|\le |(\hat{x}-\hat{y})-(x-y)|=h|b(x)-b(y)|\le hLr,
$$proving the assertion (i).

Next, we suppose that  $r=|x-y|\ge \mathcal{R}$. As we mentioned above, $K\le L$ and so  $1-2hK+h^2L^2\ge0$. Combining this fact with (c2) and the element inequality $\sqrt{1+x}\le 1+x/2$ for  $x\ge0$, we arrive at that for any $h\le 2KL^{-2}$,
\begin{align*}
\hat{r}&=\sqrt{|x-y|^2+2h \langle x-y, b(x)-b(y)\rangle +h^2|b(x)-b(y)|^2}\\
&\le r\sqrt{1-2hK+h^2L^2}\le r(1-hK+h^2L^2/2).
\end{align*} This proves (ii) due to $\beta(x,y)= \hat{r}-r$.

(iii) immediately follows from the definition \eqref{ref-basic} for the Markov coupling $((X,Y), \Pp_{x,y})$, while (iv) is also a consequence of \eqref{ref-basic}. Indeed,  for any $x,y\in \R^d$,
$$
\aligned
2\alpha(x,y)&=\Ee[(R-r)^2\I_{\{R< r\}}]\ge \frac{1}{2} \big((\hat{r}-\hat{r}\wedge\kappa)-r\big)^2\I_{\{\hat{r}-\hat{r}\wedge\kappa<r\}}\mu_{g(h)^{-1}(\hat{y}-\hat{x})_\kappa}(\rd) \\
&\ge \frac{1}{2}J_{g(h)^{-1}\kappa}\big((\hat{r}-\hat{r}\wedge\kappa)-r\big)^2\I_{\{\hat{r}-\hat{r}\wedge\kappa<r\}}.
\endaligned
$$
Now, suppose that $x,y\in \R^d$ satisfies $|x-y|\le \mathcal{R}$.
If $\hat{r}\le \kappa$, then $\hat{r}-\hat{r}\wedge \kappa=0$ and so
$$2\alpha(x,y)\ge \frac{1}{2}r^2 J_{g(h)^{-1}\kappa}.$$
If $\hat{r}\ge \kappa$, then, due to the fact that $h\le \frac{\kappa}{2L\mathcal{R}}$ implies $
\hat{r}-r\le hLr\le \frac{\kappa}{2L\mathcal{R}} Lr\le {\kappa}/{2},
$
it holds that $\hat{r}-\hat{r}\wedge\kappa=\hat{r}-\kappa\le r-\kappa/2,$ and so
$$
2\alpha(x,y)\ge \frac{\kappa^2}{8}J_{g(h)^{-1}\kappa}.
$$ The proof is complete.
\end{proof}

\begin{remark} Lemma \ref{beta} and its proof show that, for the Markov coupling $((X,Y),\Pp_{x,y})$ given by \eqref{ref-basic}, $\hat{\beta}(x,y)=\Ee_{x,y}[R-\hat{r}]=0$ for any $x,y\in \R^d$, and so $\beta(x,y)=\hat r -r$ is only determined by the coefficient $b(x)$; and that $\alpha(x,y)$ is only caused by stochastic noise $\xi$ involved in transitions.
\end{remark}

Now, according to Theorem \ref{main-TV} and Lemma \ref{beta}, we have the following statement in this part.

\begin{theorem}\label{cont-rb}
Suppose that Assumption {\bf(C)} and condition {\rm(c3)} hold, and that $\lim_{h\rightarrow0^+}h/g(h)=0$. For any  $h\in (0,2KL^{-2}\wedge L^{-1})$ such that $h/g(h)\le \kappa_0(2L\mathcal{R})^{-1}$.
Let $\rho=f(|\cdot|)$ be defined by \eqref{1f} with  $r_1=\mathcal{R}$,
$$c=\frac{64hL\mathcal{R}}{g(h)^2\kappa_0^2J_{\kappa_0}}+1,\quad a=\frac{4c(1+e^{-c\mathcal{R}})hg(h)\kappa_0 L}{J_{\kappa_0}}+1.$$
Then, there is a constant $c_*\in (0,1)$ such that for all $x,y\in\rd$,
$$
\Ee_{x,y}[\rho(X,Y)]\le (1-c_*)\rho(x,y).
$$
\end{theorem}

\begin{proof}
Let $\kappa=g(h)\kappa_0$, $r_0=g(h)\kappa_0/(1+hL)$ and $r_1=\mathcal{R}$. Then it follows from Lemma \ref{beta} that for any $h\in (0,2KL^{-2} \wedge L^{-1})$,
\begin{itemize}
\item[\rm(i)] $\inf_{r\in(0,r_0]}\underline{\pi}(r)\ge \frac{1}{2}J_{\kappa_0}>0$;

\item [\rm(ii)] $\inf_{r\in (r_0,r_1]}\underline{\alpha}(r)\ge (\frac{r_0^2}{4}\wedge\frac{g(h)^2\kappa_0^2}{16})J_{\kappa_0}=\frac{g(h)^2\kappa_0^2}{16}J_{\kappa_0}$,
    and $\sup_{r\in(0,r_1]}\overline{\beta}(r)\le hLr_1<\infty$;

\item[\rm(iii)] $\overline{\beta}(r)\le -(K-hL^2/2)hr$ for any $r\in (r_1,\infty)$.
\end{itemize}
In particular,
$$
\sup_{r\in(r_0,r_1]}\frac{4\overline{\beta}(r)_+}{\underline{\alpha}(r)}\le \frac{64hL\mathcal{R}}{g(h)^2\kappa_0^2J_{\kappa_0}}\quad \text{and}\quad \sup_{r\in(0,r_0]}\frac{\overline{\beta}(r)_+}{\underline{\pi}(r)}\le \frac{hLr_0}{J_{\kappa_0}/2}\le \frac{2hg(h)\kappa_0 L}{J_{\kappa_0}}.
$$
With those two estimates at hand, the required assertion follows from Theorem \ref{main-TV}.
\end{proof}

\begin{remark}\label{remark:Thm33}
	In many applications, the support of the distribution $\mu$ of the noise random variable $\xi$ in the chain given by \eqref{eq:chain2} is unbounded. Hence we can choose $\kappa = \infty$ in \eqref{mu_x}, i.e., we have $(x)_{\kappa} = x$ for all $x \in \R^d$. Moreover, condition {\rm (c3)} is satisfied for any $\kappa_0 > 0$. Then points \rm(iii) and \rm(iv) in the statement of Lemma \ref{beta} can be modified to state that for any $h > 0$ and any $x$, $y \in \R^d$ we have
	\begin{equation*}
	\pi(x,y) \ge \frac{1}{2}\mu_{g(h)^{-1}(\hat{y}-\hat{x})}(\rd) \quad \text{ and } \quad \alpha(x,y) \ge \frac{1}{4}r^2\mu_{g(h)^{-1}(\hat{y}-\hat{x})}(\rd).
	\end{equation*}
	Then in the proof of Theorem \ref{cont-rb} we have
	$$
	\sup_{r\in(r_0,r_1]}\frac{4\overline{\beta}(r)_+}{\underline{\alpha}(r)}\le \frac{16hL\mathcal{R}}{r_0^2 J_{g(h)^{-1}r_1}}\quad \text{and}\quad \sup_{r\in(0,r_0]}\frac{\overline{\beta}(r)_+}{\underline{\pi}(r)} \le \frac{2h L r_0}{J_{g(h)^{-1}r_0}},
	$$
	which would result in taking
	$$c=\frac{16hL\mathcal{R}}{r_0^2 J_{g(h)^{-1}r_1}}+1,\quad a=\frac{4c(1+e^{-c\mathcal{R}})h L r_0}{J_{g(h)^{-1}r_0}}+1.$$
	We can now
 further apply this to the simplified version of Theorem \ref{main-TV} as explained in Remark \ref{remark:main-TV}.
 In particular, we have $r_0 = r_1 = \mathcal{R}$, and we obtain $c = 1$ and
	\begin{equation*}
	a = \frac{4(1+e^{-\mathcal{R}})h L \mathcal{R}}{J_{g(h)^{-1}\mathcal{R}}}+1 \ge \frac{4 h L \mathcal{R}}{J_{g(h)^{-1}\mathcal{R}}}.
	\end{equation*}
	Hence if we want to track the dependence of the contractivity constant $c_*$ on the parameters such as $h$ and $\mathcal{R}$, we have to analyse the quantity $J_{g(h)^{-1}\mathcal{R}}$,
 which depends on the noise distribution $\mu$. This will be explained on the examples of $\alpha$-stable and Gaussian noises in Remark \ref{remark:noises}.
\end{remark}

As demonstrated by the following example, our results apply to very general non-isotropic distributions $\mu$.

\begin{example}
Suppose that the distribution $\mu$ of the random variable $\xi$ is given by
$$
\mu(dz)=M^{-1}\I_{\{0<z_1\le 1\}}\frac{1}{(1+|z|)^{d+\alpha}}\,dz,
$$
where $z:=(z_1,\cdots, z_d)\in\rd$,
$\alpha\in(0,2)$ and $M:=\int_{\{0<z_1\le 1\}}\frac{1}{(1+|z|)^{d+\alpha}}\, dz$. Then, we can get from the proof of \cite[Example 1.2]{LW19} that for any $\kappa\in(0,1)$,
$$
J_\kappa:=\inf_{|z|\le\kappa}(\mu\wedge(\delta_z*\mu))(\rd)\ge c((1+\kappa)^{-\alpha}-2^{-\alpha})
$$
with some constant $c:=c(\alpha,M)>0$. In particular, taking $\kappa_0=1/2$ in \eqref{kappa}, we obtain from  (i) and (ii) in the proof of Theorem \ref{cont-rb} that
$$
\inf_{r\in(0,r_0]}\underline{\pi}(r)\ge J_{\kappa_0}/2\ge \frac{c}{2}((3/2)^{-\alpha}-2^{-\alpha})>0\quad \text{and}\quad \inf_{r\in(r_0,r_1]}\underline{\alpha}(r)\ge \frac{cg(h)^2}{64}((3/2)^{-\alpha}-2^{-\alpha}).
$$
Hence, under Assumption {\bf(C)} we can apply Theorem \ref{cont-rb} to this example.
\end{example}

\begin{remark}
Let us briefly discuss an extension of Theorem \ref{cont-rb}, in which condition {\rm (c2)} in Assumption {\bf(C)} is replaced by the following weaker condition.
\begin{itemize}
 \item[{\rm (c2*)}] There are constants $M_1$, $M_2 > 0$ such that $\langle x, b(x)\rangle\le M_1-M_2|x|^2$ for all $x \in \R^d$.
 \end{itemize}
Suppose that random variable $\xi$ has the finite $\theta$-th moment for some $\theta\in(0,2]$, i.e., $\Ee[|\xi|^\theta]<\infty$.
One can check that under conditions {\rm (c1)} and {\rm (c2*)}, condition {\rm (a3*)} is satisfied for $V(x) = |x|^\theta$. Indeed, letting $h < 2M_2/L_0$
and $L_0 := 2\max \{ L^2, |b(0)|^2 \}$,  we have
 $$
\aligned
\int_{\rd}V(z)p(x,dz)&=\Ee\left[|x+hb(x)+g(h)\xi|^\theta\right]
=\Ee\left[(|x+hb(x)+g(h)\xi|^2)^{\theta/2}\right]\\
&\le (1-2hM_2+h^2L_0)^{\theta/2}|x|^\theta+(2hM_1)^{\theta/2}+g(h)^{\theta}\Ee\left[|\xi|^\theta\right]
+(h^2 L_0)^{\theta/2}
\\
&\quad+\left[\left[(2g(h))^{\theta/2}+(2hg(h))^{\theta/2}L_0^{\theta/4} \right]|x|^{\theta/2} + (2hg(h))^{\theta/2}|b(0)|^{\theta/2} \right]
\Ee\left[|\xi|^{\theta/2}\right],
\endaligned
$$
where
in the inequality above we used the fact that
$$
\aligned
|x+hb(x) +g(h)\xi|^2&=|x|^2+2h\langle x,b(x)\rangle+h^2|b(x)|^2+g(h)^2|\xi|^2
+2g(h)\langle x,\xi\rangle +2hg(h)\langle b(x),\xi\rangle\\
&\le |x|^2+2h(M_1-M_2|x|^2)+h^2 L_0|x|^2+g(h)^2|\xi|^2
+h^2 L_0
\\
&\quad +2g(h)\langle x,\xi\rangle +2hg(h)\langle b(x),\xi\rangle\\
&=(1-2hM_2+h^2L_0)|x|^2+2hM_1+
h^2 L_0+
g(h)^2|\xi|^2\\
&\quad
+2g(h)\langle x,\xi\rangle +2hg(h)\langle b(x),\xi\rangle.
\endaligned$$
This implies that {\rm (a3*)} holds for some suitable constants $\lambda$ and $C_0$. In particular, if $\theta=2$ and $\Ee[\xi]=0$, one can easily check (cf.\ \cite[Example 6.2]{EM19}) that {\rm (a3*)} holds with
\begin{equation*}
\lambda = 2hM_2 - h^2 L_0, \quad C_0 =
h^2 L_0 +
2hM_1+g(h)^2\Ee\left[|\xi|^2\right].
\end{equation*}  Moreover, under conditions {\rm (c1)}, {\rm (c2*)} and {\rm (c3)}, points {\rm (i)}, {\rm (iii)} and {\rm (iv)} in Lemma \ref{beta} remain unchanged. Hence under {\rm (c1)}, {\rm (c2*)} and {\rm (c3)}, we can apply Theorem \ref{second-TV} to prove an analogue of Theorem \ref{cont-rb} under weaker conditions. We leave the details to the reader.
\end{remark}

\subsection{Special case: $\mu$ is rotationally invariant}
In this part, we are concerned
with
the case
where
the distribution $\mu$ of the random variable $\xi$ has a density function $m(x)$ with respect to the Lebesgue measure such that
$m(x)=m(|x|)$ for all $x\in \R^d$.
For this special case, we use the following Markov coupling $((X,Y^{(1)}),\Pp_{x,y})$ of the chain $X$:
\begin{equation}\label{reflection}
\begin{cases}X=\hat{x}+g(h)z,\qquad\qquad\qquad\qquad\qquad \qquad\mu(dz);\\
Y^{(1)}=
\begin{cases}
\hat{y}+g(h)(z+g(h)^{-1}(\hat{x}-\hat{y})_\kappa),&\quad \mu_{g(h)^{-1}(\hat{y}-\hat{x})_\kappa}(dz),  \\
\hat{y}+g(h)R_{\hat{x},\hat{y}}(z),&\quad \mu(dz)-\mu_{g(h)^{-1}(\hat{y}-\hat{x})_\kappa}(dz).
\end{cases}\end{cases}
\end{equation}
Here, $\kappa>0$ is a constant fixed later, $(x)_\kappa$ and $\mu_x$ are defined in \eqref{mu_x} for any $x,y, z \in \R^d$, and
$$
R_{x,y}(z):=\begin{cases} z-\frac{2\langle x-y,z\rangle}{|x-y|^2}(x-y),\quad &x\neq y,\\
z,\quad & x=y.\end{cases}$$
See Lemma \ref{coup-2} in the appendix for the proof that $((X,Y^{(1)}),\Pp_{x,y})$ is a Markov coupling of $X$.
Note that this coupling is a generalisation of the coupling by reflection, which was used for the Gaussian noise in \cite[Section 2.4]{EM19}, to arbitrary rotationally invariant distributions
(possibly with compact support).

We first make the following assumption on the density function $m(x)$:
\begin{itemize}{\it
\item[{\rm (c4)}] The density function $m(x)=m(|x|)$ of the distribution $\mu$ satisfies that $m(r)$ is non-increasing in $(0,\infty)$, and has the finite first moment, i.e., $\int_{\R^d} |z|m(|z|) \,dz<\infty$.}
\end{itemize}

Recall that $J_{\kappa_0}$ is defined by \eqref{kappa}.

\begin{lemma}\label{beta-reflec}
Suppose that Assumption {\bf(C)} and ${\rm (c4)}$ hold. Then, for the Markov coupling given by \eqref{reflection}, we have
\item[{\rm(i)}]
$\beta(x,y)\le hL |x-y|$ for any $x,y\in\rd${\rm;}

\item[{\rm(ii)}]
$\beta(x,y)\le -(K-hL^2/2)h|x-y|$ for any $x,y\in \R^d$ with $|x-y|\ge \mathcal{R}$ and any $h< 2KL^{-2}${\rm;}

\item[\rm{(iii)}] $\pi(x,y)\ge J_{g(h)^{-1}\kappa}\I_{\{\hat{r}\le \kappa\}}$ for any $x,y\in \R^d$. In particular, for any $x,y\in \R^d$ with $|x-y|\le \kappa/(1+hL)$,
    $$\pi(x,y)\ge J_{g(h)^{-1}\kappa};$$

\item[\rm{(iv)}]$
\alpha(x,y)\ge \frac{1}{2}J_{g(h)^{-1}\kappa}\big((\hat{r}-\hat{r}\wedge\kappa)-r\big)^2\I_{\{\hat{r}-\hat{r}\wedge\kappa<r\}}
$ for any $x,y\in \R^d$.
In particular, for any $x,y\in \R^d$ with $|x-y|\le \mathcal{R}$ and any $0<h\le \kappa(2L\mathcal{R})^{-1}$,
$$
\alpha(x,y)\ge\Big(\frac{r^2}{2}\wedge\frac{\kappa^2}{8}\Big) J_{g(h)^{-1}\kappa}.
$$

  \end{lemma}
\begin{proof} The proofs of (iii) and (iv) mainly follow from the definition of $((X,Y^{(1)}),\Pp_{x,y})$ and the argument for Lemma \ref{beta}. So, we only need to verify (i) and (ii).
	The proof is similar in spirit to the proof of \cite[Lemma 2.7]{EM19}.

Recall that $\hat{r}=|\hat{x}-\hat{y}|$ and $R=|X-Y|$. According to \eqref{reflection}, for any $x,y\in \R^d$,
$$\Ee_{x,y} [R]= (\hat{r}-\hat{r}\wedge \kappa) \mu_{g(h)^{-1}(\hat{y}-\hat{x})_\kappa}(\R^d)+
\int_{\R^d} \hat{r}\left|1+2g(h)\frac{\langle \hat x-\hat y,z\rangle}{|\hat x-\hat y|^2}\right| (\mu-\mu_{g(h)^{-1}(\hat{y}-\hat{x})_\kappa})(dz).
$$
Due to the
rotational invariance
of the measure $\mu$, it suffices to assume that $\hat{x}=0$ and $\hat{y}=\hat{r}e_1$, where $e_1,\cdots,e_d$ is the canonical basis of $\rd$.
Hence, without loss of generality, we can carry out the argument to the one-dimensional case.
In particular, when $d=1$,
\begin{equation}\label{e:pppsss}
\Ee_{x,y}[R]=(\hat{r}-\hat{r}\wedge\kappa)\,\mu_{g(h)^{-1}(\hat{r}\wedge\kappa)}(\R) +\int_{\R}|\hat{r}-2g(h)z|\,(\mu-\mu_{g(h)^{-1}(\hat{r}\wedge\kappa)})(dz).
\end{equation}
For the integration in the right-hand side of the equality above, using the assumptions that $r\mapsto m(r)$ is non-increasing in $(0,\infty)$ and $z\mapsto m(z)$ is symmetric on $\R$,
we have
\begin{align*}
&\int_{\R}|\hat{r}-2g(h)z|\,(\mu-\mu_{g(h)^{-1}(\hat{r}\wedge\kappa)})(dz)\\
&=\int_{\R}|\hat{r}-2g(h)z|\left(m(z)-m(z)\wedge m(z-g(h)^{-1}(\hat{r}\wedge\kappa))\right)\,dz\\
&=\int_{-\infty}^{(2g(h))^{-1}(\hat{r}\wedge\kappa)} (\hat{r}-2g(h)z)\left[m(z)-m(z-g(h)^{-1}(\hat{r}\wedge\kappa))\right]\,dz\\
&=(\hat{r}-\hat{r}\wedge\kappa)\mu(\{|z|\le (2g(h))^{-1}(\hat{r}\wedge\kappa)\})\\
&\quad +\int_{-\infty}^{(2g(h))^{-1}(\hat{r}\wedge\kappa)}  ((\hat{r}\wedge\kappa)-2g(h)z )\left[m(z)-m(z-g(h)^{-1}(\hat{r}\wedge\kappa))\right]\,dz\\
&=(\hat{r}-\hat{r}\wedge\kappa)\mu(\{|z|\le (2g(h))^{-1}(\hat{r}\wedge\kappa)\})\\
&\quad +(2g(h))^{-1}\int_{0}^{\infty} u\left[m\big((2g(h))^{-1}((\hat{r}\wedge\kappa)-u)\big)-m\big((2g(h))^{-1}((\hat{r}\wedge\kappa)+u)\big)\right]\,du\\
&=(\hat{r}-\hat{r}\wedge\kappa)\mu(\{|z|\le (2g(h))^{-1}(\hat{r}\wedge\kappa)\})\\
&\quad +(4g(h))^{-1}\int_{-\infty}^{\infty}  u\left[m\big((2g(h))^{-1}((\hat{r}\wedge\kappa)-u)\big)-m\big((2g(h))^{-1}((\hat{r}\wedge\kappa)+u)\big)\right]\,du\\
&=(\hat{r}-\hat{r}\wedge\kappa)\mu(\{|z|\le (2g(h))^{-1}(\hat{r}\wedge\kappa)\})\\
&\quad +\frac{1}{2}\int_{-\infty}^{+\infty} \big((\hat{r}\wedge\kappa)-2g(h)z\big)\left[m(z)-m(z-g(h)^{-1}(\hat{r}\wedge\kappa))\right]\,dz\\
&=(\hat{r}-\hat{r}\wedge\kappa)\mu(\{|z|\le (2g(h))^{-1}(\hat{r}\wedge\kappa)\})\\
&\quad +
g(h)\int_{-\infty}^{\infty} \big(z-g(h)^{-1}(\hat{r}\wedge\kappa) +g(h)^{-1}(\hat{r}\wedge\kappa)\big)m(z-g(h)^{-1}(\hat{r}\wedge\kappa))\,dz\\
&=(\hat{r}-\hat{r}\wedge\kappa)\mu(\{|z|\le (2g(h))^{-1}(\hat{r}\wedge\kappa)\}) +(\hat{r}\wedge\kappa),
\end{align*}
where in the last equality we used the fact that $\int_{\R} zm(z)\,dz=0$ thanks to the symmetry property of $m(z)$ and $\int_{\R}|z|m(|z|)\,dz<\infty$.

On the other hand, also due to the assumption of $m$,
$$
\mu_{g(h)^{-1}(\hat{r}\wedge\kappa)}(\R)=\mu(\{|z|>(2g(h))^{-1}(\hat{r}\wedge\kappa)\}).
$$
Hence, putting both estimates above into \eqref{e:pppsss}, we get
$$
\Ee_{x,y}[R]
=(\hat{r}-\hat{r}\wedge\kappa)\mu(\{|z|>(2g(h))^{-1}(\hat{r}\wedge\kappa)\})
+(\hat{r}-\hat{r}\wedge\kappa)\mu(\{|z|\le (2g(h))^{-1}(\hat{r}\wedge\kappa)\}) +(\hat{r}\wedge\kappa)=\hat{r},
$$
which implies that $\beta(x,y)=\hat{r}-r$. With this at hand, the assertions (i) and (ii) then also follow from the proof of
Lemma \ref{beta}.
\end{proof}

Note that under condition (c4), there is a constant $\kappa_0>0$ such that $J_{\kappa_0}>0$, where $J_{\kappa_0}$ is defined by \eqref{kappa}. Indeed, without loss of generality, we still assume that $d=1$.  Then, for any $z\in \R$,
$$
\mu\wedge(\delta_z\ast\mu)(\R)=\int_{\R} m(u)\wedge m(u-z)\,du
=2\int_{|z|/2}^\infty m(u)\,du.
$$ This yields the desired assertion.

Combining this with Lemma \ref{beta-reflec} and following the proof of Theorem \ref{cont-rb}, we have the statement below.

\begin{theorem}
Suppose Assumption {\bf(C)} and condition {\rm(c4)} hold, and that $\lim_{h\rightarrow0^+}h/g(h)=0$.
Then there exists a constant $\kappa_0>0$ such that for any  $h\in (0,2KL^{-2}\wedge L^{-1})$ with $h/g(h)\le \kappa_0(2L\mathcal{R})^{-1}$ and for all $x,y\in\rd$,
$$
\Ee_{x,y}[\rho(X,Y^{(1)})]\le (1-c_*)\rho(x,y),
$$
where $\rho=f(|\cdot|)$
is
defined by \eqref{1f} with $r_1=\mathcal{R}$,
$$c=\frac{32hL\mathcal{R}}{g(h)^2\kappa_0^2J_{\kappa_0}}+1,\quad a=\frac{2c(1+e^{-c\mathcal{R}})hg(h)\kappa_0 L}{J_{\kappa_0}}+1$$ and the constant $c_*\in (0,1)$ is independent of $x,y\in \R^d$.
\end{theorem}

\begin{remark}\label{remark:Thm37}
	As explained in Remark \ref{remark:Thm33} for Theorem \ref{cont-rb}, when the support of the distribution $\mu$ of the random variable $\xi$ is unbounded, we can modify our bounds for $\pi(x,y)$ and $\alpha(x,y)$. Points \rm{(iii)} and \rm{(iv)} in Lemma \ref{beta-reflec} become
	\begin{equation*}
	\pi(x,y) \ge \mu_{g(h)^{-1}(\hat{y}-\hat{x})}(\rd) \quad \text{ and } \quad \alpha(x,y) \ge \frac{1}{2}r^2\mu_{g(h)^{-1}(\hat{y}-\hat{x})}(\rd)
	\end{equation*}
	for any $h > 0$ and any $x$, $y \in \R^d$, i.e., they only differ from the bounds considered in Remark \ref{remark:Thm33} by the absence of factor $\frac{1}{2}$. As a consequence, one can take
	$$c=\frac{8hL\mathcal{R}}{r_0^2 J_{g(h)^{-1}r_1}}+1,\quad a=\frac{2c(1+e^{-c\mathcal{R}})h L r_0}{J_{g(h)^{-1}r_0}}+1$$
	in Theorem \ref{main-TV}.
Furthermore,
by applying its simplified version as explained in Remark \ref{remark:main-TV}, with $r_0 = r_1 = \mathcal{R}$,
we have $c = 1$ and
\begin{equation*}
	a = \frac{2(1+e^{-\mathcal{R}})h L \mathcal{R}}{J_{g(h)^{-1}\mathcal{R}}}+1 \ge \frac{2 h L \mathcal{R}}{J_{g(h)^{-1}\mathcal{R}}}.
	\end{equation*}	
	In particular, we observe that for isotropic noise distributions $\mu$ (for which both couplings discussed above are applicable), the result based on the reflection coupling \eqref{reflection} can lead to a slightly better contractivity constant than the result based on the refined basic coupling \eqref{ref-basic}.
\end{remark}

\begin{remark}\label{remark:noises}
	Let us now consider two different noise distributions $\mu$, namely, the normal and the $\alpha$-stable distributions, with the aim of tracking the dependence of the quantity $J_{g(h)^{-1}\mathcal{R}}$
 (and hence of the contractivity constant for the corresponding chains) on parameters $h$ and $\mathcal{R}$. To this end, we will use the bounds discussed above in Remark \ref{remark:Thm37} and combine them with the simplified version of Theorem \ref{main-TV} as discussed in Remark \ref{remark:main-TV}. It is well-known that for a one-dimensional random variable $Z$ with the standard normal distribution one has the tail estimate
	\begin{equation*}
	\Pp (Z > x) \approx \exp(-x^2/2)
	\end{equation*}
	for any $x > 0$, whereas if $Z$ has the
rotationally invariant $\alpha$-stable distribution, one has
	\begin{equation*}
	\Pp (Z > x) \approx
(1+x)^{-\alpha}
	\end{equation*} for all $x>0$.
		Note that, for a rotationally invariant distribution $\mu$ with a density $m(z) = m(|z|)$ for $z \in \R^d$, we have  for any $x \in \R^d$,
	\begin{equation*}
	\begin{split}
	(\mu \wedge (\delta_x \ast \mu)) (\R^d) &= \int_{\R^d} m(z) \wedge m(z + |x|e_1) dz \\
	&=  \int_{\R^d} m((|z_1|^2+|\tilde z|^2)^{1/2})\wedge m((|z_1+|x||^2+|\tilde z|^2)^{1/2})\,dz_1\,d \tilde z\\
	&=  \int_{\{|z_1|\ge |x|/2 \}}\left( \int_{\R^{d-1}} m((|z_1|^2+|\tilde z|^2)^{1/2})\,d \tilde z \right)\,dz_1 \,,
	\end{split}
	\end{equation*}
	where $z=(z_1,\tilde z)$ and $\tilde z:=(z_2,\cdots, z_d)$. Furthermore, observe that
	\begin{equation*}
	\begin{split}
d \int_{\{|z_1|\ge |x|/2 \}}\left( \int_{\R^{d-1}} m(z)\,d \tilde z \right)\,dz_1 &\ge
\int_{\{ |z_1|\ge |x|/2 {\rm\,\, or\,\, } |z_2|\ge |x|/2 {\rm\,\, or\,\, } \ldots {\rm\, \, or\, \,} |z_d|\ge |x|/2\}} m(z)\,dz \\
&\ge \int_{\{ |z| \ge \sqrt{d} |x|/2 \}} m(z) \, dz \,.
\end{split}
	\end{equation*}
	This shows that $(\mu \wedge (\delta_x \ast \mu)) (\R^d) \ge \frac{1}{d}\int_{\{ |z| \ge \sqrt{d} |x|/2 \}} m(z) \,dz$ and hence, by applying the tail estimates above,
	\begin{equation*}
		J_{g(h)^{-1}\mathcal{R}}
 = \begin{cases}
		\Omega(d^{-1}\exp \left( -d\mathcal{R}^2/h \right))  \quad &\text{when } \mu \text{ is Gaussian}, \\
\Omega(d^{-1}(1+d^{1/2}h^{-1/\alpha}\mathcal{R})^{-\alpha})  \quad &\text{when } \mu \text{ is }\alpha\text{-stable},
		\end{cases}
	\end{equation*}
	since $g(h)=h^{1/2}$ when $\mu$ is Gaussian and $g(h)=h^{1/\alpha}$ when $\mu$ is $\alpha$-stable.
	Following Remark \ref{remark:Thm37}, we need to have
	\begin{equation}\label{eq:abound}
	a \ge \frac{2 h L \mathcal{R}}{J_{g(h)^{-1}\mathcal{R}}},
	\end{equation}
	hence, choosing $a = \frac{2 h L \mathcal{R}}{J_{g(h)^{-1}\mathcal{R}}}$,
	we see that it is of order
	\begin{equation*}
	a = \begin{cases}
\mathcal{O}\left( d h \mathcal{R} \exp \left( d \mathcal{R}^2/h \right) \right)  \quad &\text{when } \mu \text{ is Gaussian}, \\
	\mathcal{O}\left(
d h \mathcal{R}(1+d^{1/2} h^{-1/\alpha}\mathcal{R})^{\alpha} \right)  \quad &\text{when } \mu \text{ is }\alpha\text{-stable}.
	\end{cases}
	\end{equation*}
	Moreover, in the simplified version of Theorem \ref{main-TV} (cf.\ Remark \ref{remark:main-TV}) the contractivity constant $c_* = \min \{ c_1, c_3 \}$ is determined by
		\begin{equation*} c_1=\frac{a}{2(a+1+\mathcal{R}\e^{-\mathcal{R}})}\inf_{r\in(0,\mathcal{R}]}\underline{\pi}(r) \ge \frac{hL\mathcal{R}}{a+1+\mathcal{R}\e^{-\mathcal{R}}} , \quad c_3=c_0[1+(1+a)\mathcal{R}^{-1} e^{\mathcal{R}}]^{-1},
	\end{equation*}
	where in the inequality above we used \eqref{eq:abound} and the bound on $\underline{\pi}$ from Remark \ref{remark:Thm37}.
	This shows that for large $\mathcal{R}$ or for small $h$, the contractivity constant in the $\alpha$-stable case can remain much larger than the corresponding constant in the Gaussian case. In particular, $c_1 = \Omega (\mathcal{R}^{-\alpha})$ as $\mathcal{R} \to \infty$ in the $\alpha$-stable case, whereas $c_1 = \Omega (\exp(-\mathcal{R}^2))$ as $\mathcal{R} \to \infty$ in the Gaussian case.
\end{remark}

Next, we will consider the contraction of the Markov coupling  defined by \eqref{reflection} in terms of the $L^1$-Wasserstein distance.
Recall again that
$$\alpha_{l}(x,y)=\frac{1}{2}\Ee_{x,y}[(R-r)^2\I_{\{R< r+l\}}] \quad \text{and}\quad \underline{\alpha}_{l}(r)=\inf_{|x-y|=r}\alpha_{l}(x,y).
$$

\begin{lemma}\label{alpha_g(h)}
Suppose that Assumption {\bf(C)} and condition {\rm(c4)}
hold, and that $\lim_{h\rightarrow0^+}h/g(h)=0$. Then there exist constants $\varepsilon\in (0,1/4)$,
$\gamma>0$ large enough
and $c^*>0$ $($which is independent of $\gamma$ but depends on $\varepsilon)$
such that for any $\kappa, h>0$ with $\kappa\le \varepsilon g(h)/4$ and $h/g(h)\le \varepsilon (2L\mathcal{R})^{-1}$, and for any $x,y\in \rd$ with $|x-y|\in (0,\mathcal{R}]$,
$$
\alpha_{\gamma g(h)}(x,y)
\ge c^* g(h)(\hat{r}\wedge\kappa).
$$
\end{lemma}

\begin{proof}
Similarly to the proof of Lemma \ref{beta-reflec}, it suffices to consider the case $d=1$.
Let $x,y\in \R$ with $r=|x-y|\in (0,\mathcal{R}]$. Without loss of generality, we assume that $\hat{x}=x+hb(x)=0$ and $\hat{r}=\hat{y}>0$.
For $\varepsilon\in (0,1/4)$
and
$\gamma > 1 $
large enough
(which will be fixed later), by \eqref{reflection} and the properties of $m$,
\begin{align*}
&2\alpha_{\gamma g(h)}(x,y)=\Ee_{x,y}[(R-r)^2\I_{\{R\le r+\gamma g(h)\}}]\\
&\ge \varepsilon^2 g(h)^2\int_{\{r+\varepsilon g(h)\le  |\hat{r}-2g(h)z|\le r+\gamma g(h)\}}(\mu-\mu_{g(h)^{-1}(\hat{r}\wedge\kappa)})\,(dz)\\
&= \varepsilon^2 g(h)^2 \int_{\{z\le (2g(h))^{-1}(\hat{r}\wedge\kappa),\ (2g(h))^{-1}(\hat{r}-r)-\gamma/2\le  z\le (2g(h))^{-1}(\hat{r}-r)-\varepsilon/2\}} \big(m(z)-m(z-g(h)^{-1}(\hat{r}\wedge\kappa))\big)\,dz.
\end{align*}

Note that, according to (c1) in Assumption {\bf (C)} and $h/g(h)\le \varepsilon(2L\mathcal{R})^{-1}$,
$$
(2g(h))^{-1}|\hat{r}-r|\le (2g(h))^{-1}hLr\le \varepsilon/4.
$$
This, along with the condition that $\kappa\le \varepsilon g(h)/4$, yields that
$$
\aligned
\alpha_{\gamma g(h)}(x,y)&\ge
\frac{\varepsilon^2g(h)^2}{2}
\int_{(2g(h))^{-1}(\hat{r}-r)-\gamma/2}
^{(2g(h))^{-1}(\hat{r}-r)-\varepsilon/2}
\big(m(z)-m(z-g(h)^{-1}(\hat{r}\wedge\kappa))\big)\,dz\\
&\ge \frac{\varepsilon^2g(h)^2}{2}
\int_{-\gamma/2+\varepsilon/4}^{-3\varepsilon/4}
\big(m(z)-m(z-g(h)^{-1}(\hat{r}\wedge\kappa))\big)\,dz\\
&= \frac{\varepsilon^2g(h)^2}{2}
\int^{\gamma/2-\varepsilon/4}_{3\varepsilon/4}
\big(m(z)-m(z+g(h)^{-1}(\hat{r}\wedge\kappa))\big)\,dz\\
&= \frac{\varepsilon^2g(h)^2}{2}
 \int^{\gamma/2-\varepsilon/4}_{3\varepsilon/4}
\int_z^{z+g(h)^{-1}(\hat{r}\wedge\kappa)} \,\frac{-dm(s)}{ds}\,dz\\
&\ge \frac{\varepsilon^2g(h)^2}{2} \int^{\gamma/2-\varepsilon/4}_{\varepsilon} \int^s_{s-g(h)^{-1}(\hat{r}\wedge\kappa)}\,dz\, \frac{-dm(s)}{ds}\\
&=\frac{\varepsilon^2g(h)(\hat{r}\wedge\kappa)}{2} \left( m(\varepsilon)- m(\gamma/2-\varepsilon/4)\right)  \\
&\ge c_1g(h)(\hat{r}\wedge\kappa),
\endaligned
$$
where in the first equality we used the symmetry of $z\mapsto m(z)$ on $\R$ and the change of variable, in the third inequality we used Fubini's lemma and $\kappa\le \varepsilon g(h)/4$,
and the last inequality follows from the fact $\lim_{s\rightarrow\infty}m(s)=0$ (which is implied by condition (c4))
and by choosing $\varepsilon\in (0,1/4)$ small
and $\gamma>0$ large enough
so that $m(\varepsilon)- m(\gamma/2-\varepsilon/4)\ge m(\varepsilon)/2>0$.
Hence, the proof is complete. \end{proof}

Combining Lemmas \ref{beta-reflec} and \ref{alpha_g(h)} with Theorem \ref{thm-w1}, we can obtain the following statement.

\begin{theorem} Suppose that Assumption {\bf(C)} and condition {\rm(c4)}
hold, and that $\lim_{h\rightarrow0^+}h/g(h)=0$.
Let $h\in (0,2KL^{-2}\wedge(2L)^{-1})$ and
$\kappa\le \varepsilon g(h)/4$ such that $h/g(h)\le \varepsilon(2L\mathcal{R})^{-1}$
and
\begin{equation}\label{b2-g(h)}
\Big[\frac{2hL\mathcal{R}}{c^* g(h)((\mathcal{R}/2)\wedge\kappa)}+1\Big]\gamma g(h)
\le \log 2
\end{equation}
with
$\varepsilon$, $\gamma$ and
$c^*$ being the constants given in Lemma $\ref{alpha_g(h)}$. Let $\rho=f(|\cdot|)$ be the function defined by \eqref{2f} with $\Psi(r)=r$, $r_1=\mathcal{R}$, $l_0=\gamma g(h)$
and
$$
c=\frac{2hL\mathcal{R}}{c^* g(h)((\mathcal{R}/2)\wedge\kappa)}+1.
$$
Then, there is a constant $c_*\in (0,1)$ such that  for all $x,y\in \rd$,
$$
\Ee_{x,y}[\rho(X,Y^{(1)})]\le (1-c_*)\rho(x,y).
$$
\end{theorem}

\begin{proof}
According to Lemma \ref{beta-reflec},
$\overline{\beta}(r)\le hLr$ for all $r>0$, and $\overline{\beta}(r)\le -(K-hL^2/2)hr$ for all $r\in(\mathcal{R},\infty)$.
On the other hand, by Lemma \ref{alpha_g(h)},
$$\underline{\alpha}_{\gamma g(h)}(r)
\ge c^*g(h)(\hat{r}\wedge\kappa)\ge c^* g(h)\Big(\frac{r}{2}\wedge\kappa\Big)
$$ for all $r\in (0,\mathcal{R}]$,
where in the last inequality we used
the
fact that $|\hat{r}-r|\le hLr\le r/2$ due to $h\le (2L)^{-1}$. In particular,  (b1) and (b3) in Assumption {\bf(B)} are satisfied with $l_0=\gamma g(h)$,
$r_1=\mathcal{R}$ and $c_0=(K-hL^2/2)h$. Furthermore,
$$
\sup_{r\in (0,\mathcal{R}]}\frac{2\overline{\beta}(r)_+}{\underline{\alpha}_{\gamma g(h)}(r)}
\le \frac{2hL\mathcal{R}}{c^* g(h)((\mathcal{R}/2)\wedge\kappa)}.
$$
Hence, (b2) in Assumption {\bf (B)}
is satisfied because of \eqref{b2-g(h)}. Then, the desired assertion follows from
Theorem \ref{thm-w1}.
\end{proof}

Finally, we consider the $L^q$-Wasserstein distance with $q>1$. For this purpose, we need do some modifications on the Markov coupling $((X,Y^{(1)}),\Pp_{x,y})$. For $h, s, l'>0$, we consider the following Markov coupling $(X, Y^{(2)})$:
\begin{itemize}
\item[(i)] When $r=|x-y|\in(0,s]$,
\begin{equation}\label{basic-reflec1}
\begin{cases}X=\hat{x}+g(h)z,\qquad\qquad\qquad\qquad \qquad\qquad\mu(dz);\\
Y^{(2)}=
\begin{cases}
\hat{y}+g(h)(z+g(h)^{-1}(\hat{x}-\hat{y})_\kappa),&\quad \I_{\{|z|\le l', |z+g(h)^{-1}(\hat{x}-\hat{y})_\kappa|\le l'\}}\mu_{g(h)^{-1}(\hat{y}-\hat{x})_\kappa}(dz),  \\
\hat{y}+g(h)R_{\hat{x},\hat{y}}(z),&\quad \I_{\{|z|\le l', |z+g(h)^{-1}(\hat{x}-\hat{y})_\kappa|\le l'\}}(\mu-\mu_{g(h)^{-1}(\hat{y}-\hat{x})_\kappa})(dz),\\
\hat{y}+g(h)z,&\quad \I_{\{|z|> l'\, {\rm{or}}\, |z+g(h)^{-1}(\hat{x}-\hat{y})_\kappa|>l'\}}\mu(dz).
\end{cases}\end{cases}
\end{equation}
\item[(ii)] When $r\in(s,\infty)$,
\begin{equation}\label{basic-reflec2}
\begin{cases}
X=\hat{x}+g(h)z,\quad \,\,\,\,\,\mu(dz);\\
Y^{(2)}=\hat{y}+g(h)z,\quad \mu(dz).
\end{cases}\end{equation}\end{itemize}
See Lemma \ref{coup-3} in appendix for the proof that $((X,Y^{(2)}), \Pp_{x,y})$ is a Markov coupling of chain $X$.
Note that this coupling behaves like the reflection coupling \eqref{reflection} when the distance between the marginals before the jump is small (smaller than $s$) and both jump sizes are also small (smaller than $l'$), and otherwise behaves like the synchronous coupling. It is a generalisation of the coupling that was used in \cite[(2.7)]{MMS20} to obtain $L^2$ bounds in the case of the Gaussian noise.

\begin{lemma}\label{beta-br}
Suppose that Assumption {\bf(C)} and condition {\rm (c4)}
 hold,
and that $\lim_{h\rightarrow 0^+}h/g(h)=0$.
Consider the Markov coupling $((X,Y^{(2)}), \Pp_{x,y})$ with $s=\mathcal{R}$ and $l'\ge g(h)^{-1}(1+hL)\mathcal{R}+1$. It holds that for any $h<2KL^{-2}$,
\begin{itemize}
\item[\rm(i)] $\beta(x,y)\le hL|x-y|$ for any $x,y\in\rd${\rm;}
\item[\rm(ii)] $\beta(x,y)\le -(K-hL^2/2)h|x-y|$ for any $x,y\in\rd$ with $|x-y|\ge\mathcal{R}$.
\end{itemize}

Moreover, there exist constants
$\varepsilon\in(0,1/4)$, $\gamma>0$ large enough and
$c^*>0$ $($which is independent of $\gamma$ but depends on $\varepsilon)$
such that
for any $l'\ge  g(h)^{-1}(1+hL)\mathcal{R}+\gamma/2+1$,
any $\kappa, h>0$ with $\kappa\le \varepsilon g(h)/4$ and $h/g(h)\le \varepsilon(2L\mathcal{R})^{-1}$ and any $x,y\in \rd$ with $|x-y|\in (0,\mathcal{R}]$,
$$
\alpha_{\gamma g(h)}(x,y)
\ge c^* g(h)(\hat{r}\wedge\kappa).
$$
\end{lemma}

\begin{proof}
Similarly to the proof of Lemma \ref{beta-reflec}, we only need to consider the case $d=1$. Without loss of generality, we assume that $\hat y\ge \hat x$.
According to \eqref{basic-reflec1}, we know that when $r\in (0,\mathcal{R}]$,
$$
\aligned
\Ee_{x,y}[R]
 =&(\hat{r}-\hat{r}\wedge\kappa)\mu_{g(h)^{-1}(\hat{r}\wedge\kappa)}(\{|z|\le l', |z+g(h)^{-1}(\hat{x}-\hat{y})_\kappa|\le l'\}) \\
&+\int_{\R}|\hat{r}-2g(h)z|\I_{\{|z|\le l', |z+g(h)^{-1}(\hat{x}-\hat{y})_\kappa|\le l'\}}(\mu-\mu_{g(h)^{-1}(\hat{r}\wedge\kappa)})(dz) \\
&+\hat{r}\mu(\{|z|>l'\,{\rm or}\, |z+g(h)^{-1}(\hat{x}-\hat{y})_\kappa|\le l'\})\\
=&:\rm(I)+(II)+(III).
\endaligned
$$

According to the properties of the function $m(z)$ and $l'\ge g(h)^{-1}(1+hL)\mathcal{R}$, we arrive at
\begin{align*}
{\rm(II)}&=\int_{-\infty}^{(2g(h))^{-1}(\hat{r}\wedge\kappa)}(\hat{r}-2g(h)z)
\I_{\{|z|\le l',  |z+g(h)^{-1}(\hat{x}-\hat{y})_\kappa|\le l'\}}
\big(m(z)-m(z-g(h)^{-1}(\hat{r}\wedge\kappa))\big)\,dz\\
&=\int_{\R}(\hat{r}-2g(h)z)
\I_{\{|z|\le l', |z+g(h)^{-1}(\hat{x}-\hat{y})_\kappa|\le l'\}}
(\mu-\mu_{g(h)^{-1}(\hat{r}\wedge\kappa)})(dz)\\
&=(\hat{r}-\hat{r}\wedge\kappa) (\mu-\mu_{g(h)^{-1}(\hat{r}\wedge\kappa)})(\{|z|\le l', |z+g(h)^{-1}(\hat{x}-\hat{y})_\kappa|\le l'\})\\
&\quad+\int_{-l'+g(h)^{-1}(\hat{r}\wedge\kappa)}^{(2g(h))^{-1}(\hat{r}\wedge\kappa)}(\hat{r}\wedge\kappa -2g(h)z) \big(m(z)-m(z-g(h)^{-1}(\hat{r}\wedge\kappa))\big)dz\\
&=(\hat{r}-\hat{r}\wedge\kappa) (\mu-\mu_{g(h)^{-1}(\hat{r}\wedge\kappa)})(\{|z|\le l', |z+g(h)^{-1}(\hat{x}-\hat{y})_\kappa|\le l'\})\\
&\quad+(2g(h))^{-1}\int_0^{2g(h)l'-\hat{r}\wedge\kappa} u\,\big[m\big((2g(h))^{-1}(\hat{r}\wedge\kappa-u)\big)
-m\big((2g(h))^{-1}(\hat{r}\wedge\kappa+u)\big)\big]\,du\\
&=(\hat{r}-\hat{r}\wedge\kappa) (\mu-\mu_{g(h)^{-1}(\hat{r}\wedge\kappa)})(\{|z|\le l', |z+g(h)^{-1}(\hat{x}-\hat{y})_\kappa|\le l'\})\\
&\quad+(4g(h))^{-1}\int_{-2g(h)l'+\hat{r}\wedge\kappa}^{2g(h)l'-\hat{r}\wedge\kappa} u\,\big[m\big((2g(h))^{-1}(\hat{r}\wedge\kappa-u)\big)
-m\big((2g(h))^{-1}(\hat{r}\wedge\kappa+u)\big)\big]\,du\\
&=(\hat{r}-\hat{r}\wedge\kappa) (\mu-\mu_{g(h)^{-1}(\hat{r}\wedge\kappa)})(\{|z|\le l', |z+g(h)^{-1}(\hat{x}-\hat{y})_\kappa|\le l'\})\\
&\quad +(2g(h))^{-1}\int_{-2g(h)l'+\hat{r}\wedge\kappa}^{2g(h)l'-\hat{r}\wedge\kappa} u\, m\big((2g(h))^{-1}(\hat{r}\wedge\kappa-u)\big)\,du\\
&=(\hat{r}-\hat{r}\wedge\kappa) (\mu-\mu_{g(h)^{-1}(\hat{r}\wedge\kappa)})(\{|z|\le l', |z+g(h)^{-1}(\hat{x}-\hat{y})_\kappa|\le l'\})\\
&\quad+\int_{-l'+g(h)^{-1}(\hat{r}\wedge\kappa)}^{l'}(\hat{r}\wedge\kappa-2g(h)z)m(z)\,dz\\
&=(\hat{r}-\hat{r}\wedge\kappa)  (\mu-\mu_{g(h)^{-1}(\hat{r}\wedge\kappa)})(\{|z|\le l', |z+g(h)^{-1}(\hat{x}-\hat{y})_\kappa|\le l'\})\\
&\quad+(\hat{r}\wedge\kappa)\mu(\{z\in [-l'+g(h)^{-1}(\hat{r}\wedge\kappa), l']\})-2g(h)\int_{l'-{g(h)^{-1}(\hat{r}\wedge\kappa)}}^{l'}zm(z)\,dz\\
&\le (\hat{r}-\hat{r}\wedge\kappa) ((\mu-\mu_{g(h)^{-1}(\hat{r}\wedge\kappa)})\{|z|\le l', |z+g(h)^{-1}(\hat{x}-\hat{y})_\kappa|\le l'\}\\
&\quad+(\hat{r}\wedge\kappa)\mu(\{|z|\le l', |z+g(h)^{-1}(\hat{x}-\hat{y})_\kappa|\le l'\}).
\end{align*}

Combining both estimates above, we get that
$
\Ee_{x,y}[R]\le \hat{r} $ for $r\in (0,\mathcal{R}]$.

When $r\in(\mathcal{R},\infty)$, by \eqref{basic-reflec2}, it is clear that $R=\hat{r}$. Therefore, we have $\beta(x,y)\le \hat{r}-r$ for all $x,y\in\R$. This proves the first assertion.

\medskip

Next, we turn to the proof of the second assertion. From condition (c4) and the fact that $\lim_{s\rightarrow\infty}m(s)=0$, we can choose $\varepsilon\in(0,1/4)$ small enough and
$\gamma > 1$
large enough so that $m(\varepsilon)-m(\gamma/2-\varepsilon/4)\ge m(\varepsilon)/2>0$.
Since $h/g(h)\le \varepsilon(2L\mathcal{R})^{-1}$ and $l'\ge  g(h)^{-1}(1+hL)\mathcal{R}+\gamma/2+1$, for all $r\in(0,\mathcal{R}]$,
 $$[-\gamma/2-1,0]\subset[ -l'+g(h)^{-1}(\hat r\wedge \kappa),l']=\{z\in \R: |z|\le l', |z+g(h)^{-1}(\hat{x}-\hat{y})_\kappa|\le l'\}$$
and
$$
[(2g(h))^{-1}(\hat{r}-r)-\gamma/2,(2g(h))^{-1}(\hat{r}-r)-\varepsilon/2]
\subset[-\gamma/2-\varepsilon/4,-\varepsilon/4]\subset [-\gamma/2-1,0].
$$
Then, following the proof of Lemma \ref{alpha_g(h)}, we have
\begin{align*}
\alpha_{\gamma g(h)}(x,y)&=\frac{1}{2}\Ee_{x,y}[(R-r)^2\I_{\{R\le r+\gamma g(h)\}}]\\
&\ge \frac{\varepsilon^2g(h)^2}{2}
\int_{\{r+\varepsilon g(h)\le |\hat{r}-2g(h)z|\le r+\gamma g(h)\}}
\I_{\{|z|\le l', |z+g(h)^{-1}(\hat{x}-\hat{y})_\kappa|\le l'\}}
\big(\mu-\mu_{g(h)^{-1}(\hat{r}\wedge\kappa)}\big)(dz)\\
&=\frac{\varepsilon^2g(h)^2}{2}
\int_{(2g(h))^{-1}(\hat{r}-r)-\gamma/2}^{(2g(h))^{-1}(\hat{r}-r)-\varepsilon/2}
\big(m(z)-m(z-g(h)^{-1}(\hat{r}\wedge\kappa))\big)\,dz\\
&\ge c^* g(h)(\hat{r}\wedge\kappa).
\end{align*}
The proof is complete.
\end{proof}

We note that, under Assumption {\bf (C)} and for the step size $h<2KL^{-2}$,
 it is easy to see that the Markov coupling $((X,Y^{(2)}), \Pp_{x,y})$ defined by \eqref{basic-reflec1} and \eqref{basic-reflec2} satisfies $|X-Y^{(2)}|\le |x-y|+l$, where \begin{equation}\label{e:lll--}l=hLs+\kappa\vee (2g(h)l').\end{equation}

With this estimate and Lemma \ref{beta-br} at hand, we can follow the proof of Theorem  \ref{thm-wq} to get the following assertion.

\begin{theorem}\label{thm:Lp} Consider the Markov coupling $((X,Y^{(2)}), \Pp_{x,y})$ with $s=\mathcal{R}$ and $l'\ge (2g(h))^{-1}(1+hL)\mathcal{R}+\gamma/2+1$, where
$\gamma$ 	
is the constant given in Lemma $\ref{beta-br}$.
Suppose that Assumption {\bf(C)} and condition {\rm(c4)}
hold, and that $\lim_{h\rightarrow0^+}h/g(h)=0$.
Let
$h\in (0,2KL^{-2}\wedge(2L)^{-1})$ and $\kappa\le \varepsilon g(h)/4$ such that $h/g(h)\le \varepsilon(2L\mathcal{R})^{-1}$
and
$$
\Big[\frac{2hL\mathcal{R}}{c^* g(h)((\mathcal{R}/2)\wedge\kappa)}+1\Big]\gamma g(h)
\le \log 2
$$
with
$\varepsilon$, $\gamma$ and
$c^*$ being the constants given in Lemma $\ref{beta-br}$.
Let $\rho=f(|\cdot|)$ be the function defined by \eqref{4fff} with $\Psi(r)=r$, $r_1=\mathcal{R}$, $l_0=\gamma g(h)$,
$l$ given by \eqref{e:lll--} and
$$
c=\frac{2hL\mathcal{R}}{c^* g(h)((\mathcal{R}/2)\wedge\kappa)}+1.
$$
Then,  there exists a constant $c_*\in(0,1)$ such that for all $x,y\in \rd$,
$$
\Ee_{x,y}[\rho(X,Y^{(2)})]\le (1-c_*)\rho(x,y).
$$
\end{theorem}
\begin{remark}
	Note that Theorem \ref{thm:Lp} is an extension of Theorem 2.1 in \cite{MMS20}, where a similar contraction result was proven, but only in the $L^2$-Wasserstein distance and only for chains with the Gaussian noise. One of the main motivations for considering such contractions in \cite{MMS20} was the analysis of Multi-level Monte Carlo (MLMC) methods based on chains \eqref{eq:chain2} in the Gaussian case, for approximating integrals of Lipschitz functions with respect to invariant measures of Langevin SDEs, see Theorem 1.7 therein. By following the analysis of MLMC in Subsection 2.5 in \cite{MMS20}, it is easy to see that by employing our Theorem \ref{thm:Lp}, it is possible to extend Theorem 1.7 in \cite{MMS20} from Lipschitz functions to all functions with a polynomial growth. Another possible extension would be the analysis of MLMC methods based on discretisations of SDEs with L\'{e}vy noises. We leave the details for future work.
\end{remark}

\section{Appendix}
\begin{lemma}\label{coup-1} $((X,Y),\Pp_{x,y})$  defined by \eqref{ref-basic} is a Markov coupling of the chain $X$. \end{lemma}

\begin{proof}
Fix $x,y\in \rd$, and recall that $\hat{x}=x+hb(x)$ and $\hat{y}=y+hb(y)$. By \eqref{ref-basic}, it suffices to prove that the distribution of the random variable $g(h)^{-1}(Y-\hat{y})$ is $\mu$. Indeed, for any $A\in \mathscr{B}(\rd)$,
\begin{align*}
\Pp_{x,y}(g(h)^{-1}(Y-\hat{y})\in A)=&\frac{1}{2}\mu_{g(h)^{-1}(\hat{y}-\hat{x})_\kappa}
(A-g(h)^{-1}(\hat{x}-\hat{y})_\kappa)
+\frac{1}{2}\mu_{g(h)^{-1}(\hat{x}-\hat{y})_\kappa}
(A-g(h)^{-1}(\hat{y}-\hat{x})_\kappa)\\
&+\bigg(\mu-\frac{1}{2}\mu_{g(h)^{-1}(\hat{y}-\hat{x})_\kappa}
-\frac{1}{2}\mu_{g(h)^{-1}(\hat{x}-\hat{y})_\kappa}\bigg)(A).
\end{align*}
According to
\begin{equation}\label{delta_-v}
(\delta_{-v}*\mu_{v})(dz)=\mu_{-v}(dz)\quad \text{for all }v\in \R^d,
\end{equation}
we find that
\begin{equation}\label{mu_g(h)}
\mu_{g(h)^{-1}(\hat{y}-\hat{x})_\kappa}
(A-g(h)^{-1}(\hat{x}-\hat{y})_\kappa)
= \delta_{g(h)^{-1}(\hat{x}-\hat{y})_\kappa}*\mu_{g(h)^{-1}(\hat{y}-\hat{x})_\kappa}(A)=\mu_{g(h)^{-1}(\hat{x}-\hat{y})_\kappa}(A).
\end{equation}
Similarly,
$$
\mu_{g(h)^{-1}(\hat{x}-\hat{y})_\kappa}
(A-g(h)^{-1}(\hat{y}-\hat{x})_\kappa)=\mu_{g(h)^{-1}(\hat{y}-\hat{x})_\kappa}(A).
$$
Hence, $\Pp_{x,y}(g(h)^{-1}(Y-\hat{y})\in A)=\mu(A)$ for all $A\in \mathscr{B}(\rd)$. This completes the proof.
\end{proof}

\begin{lemma}\label{coup-2} $((X,Y^{(1)}),\Pp_{x,y})$  defined by \eqref{reflection} is a Markov coupling of the chain $X$. \end{lemma}

\begin{proof}
Similarly as in the proof of Lemma \ref{coup-1}, we only need to verify that the distribution of the random variable $g(h)^{-1}(Y^{(1)}-\hat{y})$ is $\mu$. For any $A\in \mathscr{B}(\rd)$, by \eqref{reflection},
$$
\Pp_{x,y}(g(h)^{-1}(Y^{(1)}-\hat{y})\in A)=\mu_{g(h)^{-1}(\hat{y}-\hat{x})_\kappa}
(A-g(h)^{-1}(\hat{x}-\hat{y})_\kappa)
+\big(\mu-\mu_{g(h)^{-1}(\hat{y}-\hat{x})_\kappa}\big)(R^{-1}_{\hat{x},\hat{y}}(A)),
$$
where $R^{-1}_{\hat{x},\hat{y}}(A)=\{z\in\rd:\ R_{\hat{x},\hat{y}}(z)\in A\}$. Since $R_{\hat{x},\hat{y}}(z)=R^{-1}_{\hat{x},\hat{y}}(z)$, $|R_{\hat{x},\hat{y}}(z)|=|z|$ and $m(z)=m(|z|)$ for all $z\in \rd$, we have
$
\mu(R^{-1}_{\hat{x},\hat{y}}(A))=\mu(R_{\hat{x},\hat{y}}(A))=\mu(A).
$ Moreover,
$$
\aligned
\mu_{g(h)^{-1}(\hat{y}-\hat{x})_\kappa}(R^{-1}_{\hat{x},\hat{y}}(A))
&=\int_A m(R^{-1}_{\hat{x},\hat{y}}(z))
\wedge m\big(R^{-1}_{\hat{x},\hat{y}}(z)-g(h)^{-1}(\hat{y}-\hat{x})_\kappa\big)\,dz\\
&=\int_A m(z)
\wedge m\big(z-R_{\hat{x},\hat{y}}(g(h)^{-1}(\hat{y}-\hat{x})_\kappa)\big)\,dz\\
&=\int_A m(z)
\wedge m\big(z-g(h)^{-1}(\hat{x}-\hat{y})_\kappa\big)\,dz\\
&=\mu_{g(h)^{-1}(\hat{x}-\hat{y})_\kappa}(A),
\endaligned
$$
where in the third equality we used the fact that
\begin{equation}\label{R-g(h)}
R_{\hat{x},\hat{y}}(g(h)^{-1}(\hat{y}-\hat{x})_\kappa)=g(h)^{-1}(\hat{x}-\hat{y})_\kappa.
\end{equation}
Therefore, $\Pp_{x,y}(g(h)^{-1}(Y^{(1)}-\hat{y})\in A)=\mu(A)$ by \eqref{mu_g(h)}.
\end{proof}

\begin{lemma}\label{coup-3} $((X,Y^{(2)}),\Pp_{x,y})$ defined by \eqref{basic-reflec1} and \eqref{basic-reflec2} is a Markov coupling of the chain $X$.  \end{lemma}

\begin{proof}
Fix $h,s,l',\kappa>0$.
When $r=|x-y|\in(s,\infty)$, $((X,Y^{(2)}),\Pp_{x,y})$ defined by \eqref{basic-reflec2} is a synchronous coupling. Hence, we only need to consider the case
where $r\in (0,s]$.

For any $A\in\mathscr{B}(\rd)$, by \eqref{basic-reflec1},
$$
\aligned
\Pp_{x,y}(g(h)^{-1}(Y^{(2)}-\hat{y})\in A)
&=\int_{A-g(h)^{-1}(\hat{x}-\hat{y})_\kappa}\I_{\{|z|\le l', |z+g(h)^{-1}(\hat{x}-\hat{y})_\kappa|\le l'\}}\mu_{g(h)^{-1}(\hat{y}-\hat{x})_\kappa}(dz)\\
&\quad +\int_{R^{-1}_{\hat{x},\hat{y}}(A)}\I_{\{|z|\le l', |z+g(h)^{-1}(\hat{x}-\hat{y})_\kappa|\le l'\}}(\mu-\mu_{g(h)^{-1}(\hat{y}-\hat{x})_\kappa})(dz)\\
&\quad+\int_A \I_{\{|z|> l'\, {\rm{or}}\, |z+g(h)^{-1}(\hat{x}-\hat{y})_\kappa|>l'\}}\,\mu(dz)\\
&=:(\rm I)+(\rm II)+(\rm III).
\endaligned
$$

It follows from  \eqref{delta_-v} that
$$
{\rm(I)}= \int_{A}\I_{\{|u-g(h)^{-1}(\hat{x}-\hat{y})_\kappa|\le l', |u|\le l'\}}\mu_{g(h)^{-1}(\hat{x}-\hat{y})_\kappa}(du).
$$
On the other hand, due to the rotational invariance
of $\mu$, the properties of $R_{\hat{x},\hat{y}}$ and \eqref{R-g(h)}, we have $$
\aligned
{\rm (II)}&=\int_A
\I_{\{|R_{\hat{x},\hat{y}}^{-1}(u)|\le l', |R_{\hat{x},\hat{y}}^{-1}(u)+g(h)^{-1}(\hat{x}-\hat{y})_\kappa|\le l'\}}(\mu-\mu_{g(h)^{-1}(\hat{y}-\hat{x})_\kappa})(dR_{\hat{x},\hat{y}}^{-1}(u))\\
&=\int_A
\I_{\{|u|\le l', |u-g(h)^{-1}(\hat{x}-\hat{y})_\kappa|\le l'\}}\mu(dR_{\hat{x},\hat{y}}^{-1}(u))\\
&\quad-\int_A
\I_{\{|u|\le l', |u-g(h)^{-1}(\hat{x}-\hat{y})_\kappa|\le l'\}}\big[\mu(dR_{\hat{x},\hat{y}}^{-1}(u))\wedge \mu(d(R_{\hat{x},\hat{y}}^{-1}(u)-g(h)^{-1}(\hat{y}-\hat{x})_\kappa))\big]\\
&=\int_A
\I_{\{|u|\le l', |u-g(h)^{-1}(\hat{x}-\hat{y})_\kappa|\le l'\}}\mu(du)
-\int_A
\I_{\{|u|\le l', |u-g(h)^{-1}(\hat{x}-\hat{y})_\kappa|\le l'\}}\mu_{g(h)^{-1}(\hat{x}-\hat{y})_\kappa}(du).
\endaligned
$$

Therefore,
$$
\Pp_{x,y}(g(h)^{-1}(Y^{(2)}-\hat{y})\in A)=\mu(A)\quad \text{for all }A\in \mathscr{B}(\rd).
$$
The proof is complete.
\end{proof}

\bigskip

\noindent{\bf Acknowledgement.} \rm
The research of Lu-Jing Huang is supported by the National Natural Science Foundation of China (No.\ 11901096) and the National Natural Science Foundation of Fujian (No.\ 2020J05036).
The research of Jian Wang is supported by the National Natural Science Foundation of China (Nos.\ 11831014 and 12071076),
the Program for Probability and Statistics:
Theory and Application (No.\ IRTL1704),
and the Program for Innovative Research Team in Science and Technology
in Fujian Province University (IRTSTFJ).

\end{document}